\documentclass[11pt,review]{elsarticle}
\usepackage[utf8]{inputenc}
\usepackage{graphicx}
\graphicspath{ {./images/} }
\usepackage{mathtools}
\allowdisplaybreaks
\newcommand\scalemath[2]{\scalebox{#1}{\mbox{\ensuremath{\displaystyle #2}}}}
\usepackage{srcltx}
\usepackage{eurosym}
\usepackage{amssymb}
\usepackage[all,arc]{xy}
\usepackage{enumerate}
\usepackage{mathrsfs}
\usepackage{pst-node}
\usepackage{amsmath,amsthm}
\usepackage{mathtools}
\usepackage[T1]{fontenc}
\usepackage{fourier}
\usepackage{baskervald}
\usepackage[doublespacing]{setspace}
\usepackage[titletoc,toc,page]{appendix}

\usepackage[colorlinks,plainpages=true,pdfpagelabels,hypertexnames=true,colorlinks=true,pdfstartview=FitV,linkcolor=blue,citecolor=red,urlcolor=black,pagebackref]{hyperref}
\PassOptionsToPackage{unicode}{hyperref}
\PassOptionsToPackage{naturalnames}{hyperref}

\newcommand*\bb{\mathbb}

\newcommand *\w{^\wedge}

\newcommand*\de{\partial}

\makeatletter
\newcommand{\vo}{\vec{o}\@ifnextchar{^}{\,}{}}
\makeatother

\def\YYint#1#2#3{{\setbox0=\hbox{$#1{#2#3}{\iint}$}
    \vcenter{\hbox{$#2#3$}}\kern-.50\wd0}}

\def\XXint#1#2#3{{\setbox0=\hbox{$#1{#2#3}{\int}$}
    \vcenter{\hbox{$#2#3$}}\kern-.50\wd0}}

\makeatletter
\def\namedlabel#1#2{\begingroup
   \def\@currentlabel{#2}%
   \label{#1}\endgroup
}
\makeatother
\makeatletter
\newcommand{\rmh}[1]{\mathpalette{\raisem@th{#1}}}
\newcommand{\raisem@th}[3]{\hspace*{-1pt}\raisebox{#1}{$#2#3$}}
\makeatother

\newcommand{\descref}[2]{\hyperref[#1]{\textnormal{\textcolor{black}{(}\textcolor{blue}{\bf #2}\textcolor{black}{)}}}}

\newcommand{\dref}[2]{\hyperref[#1]{\textcolor{black}{(}\textcolor{blue}{\bf #2}\textcolor{black}{)}}}

\newcommand\RR{\mathbb{R}}

\newcommand{\ve}{\varepsilon}

\newcommand{\Om}{\Omega}

\newcommand{\iprod}[2]{\langle #1 \ ,  #2\rangle}

\usepackage[ddmmyyyy]{datetime}
\usepackage[margin=2cm]{geometry}
\makeatletter
\g@addto@macro\normalsize{%
  \setlength\abovedisplayskip{2pt}
  \setlength\belowdisplayskip{2pt}
  \setlength\abovedisplayshortskip{4pt}
  \setlength\belowdisplayshortskip{4pt}
}
\numberwithin{equation}{section}
\everymath{\displaystyle}
\usepackage[capitalize,nameinlink]{cleveref}
\crefname{section}{Section}{Sections}
\crefname{subsection}{Subsection}{Subsections}
\crefname{condition}{Condition}{Conditions}
\crefname{hypothesis}{Hypothesis}{Hypothesis}
\crefname{assumption}{Assumption}{Assumptions}
\crefname{lemma}{Lemma}{Lemmas}
\crefname{claim}{Claim}{Claims}
\crefname{remark}{Remark}{Remarks}

\crefformat{equation}{\textup{#2(#1)#3}}
\crefrangeformat{equation}{\textup{#3(#1)#4--#5(#2)#6}}
\crefmultiformat{equation}{\textup{#2(#1)#3}}{ and \textup{#2(#1)#3}}
{, \textup{#2(#1)#3}}{, and \textup{#2(#1)#3}}
\crefrangemultiformat{equation}{\textup{#3(#1)#4--#5(#2)#6}}%
{ and \textup{#3(#1)#4--#5(#2)#6}}{, \textup{#3(#1)#4--#5(#2)#6}}%
{, and \textup{#3(#1)#4--#5(#2)#6}}

\Crefformat{equation}{#2Equation~\textup{(#1)}#3}
\Crefrangeformat{equation}{Equations~\textup{#3(#1)#4--#5(#2)#6}}
\Crefmultiformat{equation}{Equations~\textup{#2(#1)#3}}{ and \textup{#2(#1)#3}}
{, \textup{#2(#1)#3}}{, and \textup{#2(#1)#3}}
\Crefrangemultiformat{equation}{Equations~\textup{#3(#1)#4--#5(#2)#6}}%
{ and \textup{#3(#1)#4--#5(#2)#6}}{, \textup{#3(#1)#4--#5(#2)#6}}%
{, and \textup{#3(#1)#4--#5(#2)#6}}

\crefdefaultlabelformat{#2\textup{#1}#3}

\usepackage{enumitem}
\newlist{steps}{enumerate}{1}
\setlist[steps, 1]{label = \textcolor{Cerulean}{Step \arabic*:}}
\makeatletter
\def\ps@pprintTitle{%
	\let\@oddhead\@empty
	\let\@evenhead\@empty
	\def\@oddfoot{}%
	\let\@evenfoot\@oddfoot}
\makeatother
\DeclarePairedDelimiterX{\inp}[2]{\langle}{\rangle}{#1, #2}
\newcommand{\norm}[1]{\left\lVert#1\right\rVert}

\newtheorem{theorem}{Theorem}[section]
\newtheorem{lemma}[theorem]{Lemma}

\newtheorem{corollary}[theorem]{Corollary}
\newtheorem{proposition}[theorem]{Proposition}

\newtheorem{definition}[theorem]{Definition}
\newtheorem{remark}[theorem]{Remark}        

\numberwithin{equation}{section}

\usepackage{doi}
\begin{document}

\begin{frontmatter}
\title{Existence of variational solutions to nonlocal evolution equations via convex minimization}
\author[myaddress]{Harsh Prasad}
\ead{harsh@tifrbng.res.in}

\author[myaddress]{Vivek Tewary\corref{mycorrespondingauthor}}
\ead{vivektewary@gmail.com and vivek2020@tifrbng.res.in}


\address[myaddress]{Tata Institute of Fundamental Research, Centre for Applicable Mathematics, Bangalore, Karnataka, 560065, India}
\cortext[mycorrespondingauthor]{Corresponding author}
\begin{abstract}
    We prove existence of variational solutions for a class of nonlocal evolution equations whose prototype is the double phase equation 
    		    \begin{align*}
        \de_t u + P.V.\int_{\mathbb{R}^N}& \frac{|u(x,t)-u(y,t)|^{p-2}(u(x,t)-u(y,t))}{|x-y|^{N+ps}}\\
        &+a(x,y)\frac{|u(x,t)-u(y,t)|^{q-2}(u(x,t)-u(y,t))}{|x-y|^{N+qs}} \,dy = 0.
   		    \end{align*}	 
   	 
    The approach of minimization of parameter-dependent convex functionals over space-time trajectories requires only appropriate convexity and coercivity assumptions on the nonlocal operator. As the parameter tends to zero, we recover variational solutions. Under further growth conditions, these variational solutions are global weak solutions. Further, this provides a direct minimization approach to approximation of nonlocal evolution equations.
\end{abstract}

\end{frontmatter}
\begin{singlespace}
\tableofcontents
\end{singlespace}

\section{Introduction}

\subsection{The problem}
In \cite{bogeleinExistenceEvolutionaryVariational2014}, B\"ogelein et al prove existence of variational solutions for parabolic systems of the form
\begin{align*}
\de_t u -\text{div}(D_\xi f(x,u,Du)) = 0 \text{ in } \Omega_{\infty}\\
u = u_0 \text{ on } \de_p\Omega_{\infty} 
\end{align*} where the convex function $f:\RR\times\RR^N\times\RR^{Nn}\to\RR$ satisfies a growth condition from below, viz,
\begin{align*}
    C|\xi|^p\leq f(x,z,\xi),
\end{align*} for a.e. $x\in\Omega,\,z\in\RR^N,\,\xi\in\mathbb{R}^{Nn}$. The notation $\Omega_\infty$ stands for $\Omega\times (0,\infty)$. In particular, no growth condition from above is assumed. This is useful, for example, in the context of parabolic equations with $p,q$ growth conditions for $p\leq q$, where weak solutions may not exist without imposing further conditions on the gap of $p$ and $q$~\cite{bogeleinParabolicSystemsQGrowth2013}. In this paper, we aim to extend the framework of variational solutions to parabolic fractional equations with time independent initial and boundary data
\begin{align}\label{maineq}
\de_t u &+ P.V.\int_{\mathbb{R}^N} \frac{\text{D}_{\xi}H(x,y,u(x,t)-u(y,t))}{|x-y|^N} \,dy = 0 \text{ in } \Omega_{\infty}\nonumber\\
u &= u_0 \text{ on } (\mathbb{R}^N\setminus\Omega)\times(0,\infty)\cup {\Omega}\times\{0\}
\end{align}
where $\Omega$ is an open bounded subset of $\bb{R}^N$, and $\Omega_{\infty} = \Omega\times (0,\infty)$. The function $H=H(x,y,\xi)$ satisfies the following structure condition
\begin{align}
    &\label{eq:bound_below_H} H(x,y,\xi) \geq A\left(\frac{|\xi|}{|x-y|^s}\right)^p\\ 
    &\label{eq: cvx_H} H \text{ is a  Caratheodory function which is convex in the variable } \xi.
\end{align}
We assume $1<p<\infty$, $s \in (0,1)$ and $A>0$.
These structure conditions admit a variety of problems with non-standard growth such as
\begin{itemize}
    \item $H(x,y,\xi) = \left(\frac{|\xi|}{|x-y|^s}\right)^p+a(x,y)\left(\frac{|\xi|}{|x-y|^r}\right)^q$, for $a(x,y)\geq 0$ for $1<p<q$ and $r,s\in (0,1)$. 
    \item $H(x,y,\xi) = \left(\frac{|\xi|}{|x-y|^s}\right)^p\log\left(1+\left(\frac{|\xi|}{|x-y|^s}\right)\right).$
    \item $H(x,y,\xi) = \left(\frac{|\xi|}{|x-y|^s}\right)^{a+b\sin\left(\log\,\log\left(\frac{|\xi|}{|x-y|^s}\right)\right)}$.
\end{itemize}
In particular, we would like to emphasize that in the double phase case, we obtain existence of variational solutions \textit{without} any restrictions on the gap $q-p$.  
\subsection{Background}

There has been a great surge in the study of regularity theory of fractional $p$-Laplace equations and their parabolic counterparts. However, a theory for parabolic fractional equations with non-standard growth, particularly with a double phase character, requires the development of an existence theory which can contend with the unbalanced growth condition given the possibility of the appearance of Lavrentiev phenomenon. The variational framework introduced in~\cite{bogeleinExistenceEvolutionaryVariational2014} seems to be the most suitable for this purpose. 

The motivation for the notion of variational solutions in~\cite{bogeleinExistenceEvolutionaryVariational2014} came from~\cite{lichnewskyPseudosolutionsTimedependentMinimal1978} where a variational notion of solutions was introduced for an evolutionary minimal surface equation. Further, it is related to the De Giorgi conjectures on the construction of weak solutions to hyperbolic equations as a limit of minimizers of certain functionals \cite{degiorgiConjecturesConcerningEvolution1996}. Some forms of these conjectures were proved in~\cite{stefanelliGiorgiConjectureElliptic2011,serraNonlinearWaveEquations2012}. Variational notions are also useful in the realm of stochastic partial differential equations, for example, see~\cite{menovschikovExtendedVariationalTheory2021,scarpaStochasticPDEsConvex2021}.

The regularity theory of $p,q$ growth problems was started by Marcellini in a series of novel papers~\cite{marcelliniRegularityMinimizersIntegrals1989,marcelliniRegularityExistenceSolutions1991,marcelliniRegularityEllipticEquations1993,marcelliniRegularityScalarVariational1996}. There is a large body of work dealing with problems of $(p,q)$-growth as well as other nonstandard growth problems, for which we point to the surveys in \cite{marcelliniRegularityGeneralGrowth2020,mingioneRecentDevelopmentsProblems2021}.

Coming back to the fractional $p$-Laplace equation, in the elliptic case, regularity theory of fractional $p$-Laplace equations has been studied extensively. For example, local boundedness and H\"older regularity in the framework of De Giorgi-Nash-Moser theory was worked out in~\cite{dicastroLocalBehaviorFractional2016} and \cite{cozziRegularityResultsHarnack2017}. Moreover, explicit higher regularity of the gradient is obtained in~\cite{brascoHigherSobolevRegularity2017}. On the other hand, explicit H\"older regularity of the solutions is obtained in \cite{brascoHigherHolderRegularity2018}. Higher integrability by a nonlocal version of Gehring's Lemma was proved in~\cite{kuusiFractionalGehringLemma2014}. For equations of nonstandard growth, the relevant works are \cite{scottSelfimprovingInequalitiesBounded2020,byunOlderRegularityWeak2021,chakerRegularitySolutionsAnisotropic2020,chakerRegularityEstimatesFractional2021,chakerNonlocalOperatorsSingular2020,chakerLocalRegularityNonlocal2021,chakerRegularityNonlocalProblems2021}. For the case of linear equations, i.e., $p=2$, we refer to \cite{caffarelliDriftDiffusionEquations2010,caffarelliRegularityTheoryParabolic2011,laraRegularitySolutionsNon2014,chang-laraRegularitySolutionsNonlocal2014}.

In the case of parabolic counterparts of the fractional $p$-Laplace equations, local boundedness was proved in~\cite{stromqvistLocalBoundednessSolutions2019}. Local boundedness and H\"older regularity has been proved in~\cite{dingLocalBoundednessHolder2021}. Explicit H\"older regularity has been obtained in~\cite{brascoContinuitySolutionsNonlinear2021}. 

We believe this is the first work to deal with fractional evolutionary equations exhibiting unbalanced growth. In the companion article~\cite{prasadLocalBoundednessVariational2021}, we prove local boundedness for parabolic minimizers of $(p,q)$-fractional parabolic equations following~\cite{cozziRegularityResultsHarnack2017}. In another article \cite{ghoshExistenceVariationalSolutions2022}, existence of nonnegative variational solutions for doubly nonlinear nonlocal parabolic equations is proved.

\subsection{Definition}

\begin{definition}
    Let $\Om$ be an open bounded subset of $\RR^N$. Suppose that $H$ satisfies the assumptions~\cref{eq:bound_below_H} and~\cref{eq: cvx_H} and let the time-independent Cauchy-Dirichlet data $u_0$ satisfy
    \begin{align}\label{datahypo}
        u_0\in W^{s,p}(\mathbb{R}^N), u_{|\Omega}\in L^2(\Omega)\mbox{ and }\iint\limits_{\mathbb{R}^N\times\mathbb{R}^N}\frac{H(x,y,u_0(x)-u_0(y))}{|x-y|^N}\,dx\,dy<\infty.
    \end{align}
     By a variational solution to~\cref{maineq} we mean a function $u$ such that for all time $T>0$, $u\in L^p(0,T;W^{s,p}(\mathbb{R}^N))\cap C^0(0,T;L^2(\Om))$, $u-u_0\in L^p(0,T;W^{s,p}_0(\Om))$ and
     \begin{align}
         \label{defvar}
         \int_{0}^{T} \int_{\Om}{\de_t v}\,{(v-u)}\,dx\,dt & + \int_{0}^{T}\iint\limits_{\RR^N\times\RR^N} \frac{H(x,y,v(x,t)-v(y,t))-H(x,y,u(x,t)-u(y,t))}{|x-y|^N}\,dx\,dy\,dt\nonumber\\
         &\geq \frac{1}{2}\norm{(v-u)(\cdot,\tau)}^2_{L^2(\Om)}-\frac{1}{2}\norm{(v-u)(\cdot,0)}^2_{L^2(\Om)},
     \end{align} for all $v\in L^p(0,T;W^{s,p}(\mathbb{R}^N))$ and $\de_t v\in L^2(0,T;\Om)$ such that $v-u_0\in L^p(0,T;W_0^{s,p}(\Om))$.
\end{definition}
\begin{remark}
In our definition of variational solution, both the solution and the comparison map have to match on $\Omega^c\times(0,T)$ and since we have assumed that the data on $\Omega^c$ is in $W^{s,p}(\bb{R}^N)$ we may cancel integrals of $H$ over $\Omega^c\times\Omega^c$ on both sides to obtain the following equivalent form of the variational inequality: 
\[
\begin{split}
    \int_0^T \iint\limits_{C_{\Omega}}\frac{H(x,y,u(x,t)-u(y,t))}{|x-y|^N} \,dx \,dy \,dt \leq \int_0^T \iint\limits_{C_{\Omega}}\frac{H(x,y,v(x,t)-v(y,t))}{{|x-y|^N}} \,dx \,dy \,dt \\ + \int_0^T \int_{\Omega}\de_tv\cdot(v-u) + \frac{1}{2}\norm{v(\cdot,0) - u_0(\cdot)}_{L^2(\Omega)}^2 - \frac{1}{2}\norm{v(\cdot,T) - u(\cdot,T)}_{L^2(\Omega)}^2
\end{split}
\]
where $C_{\Omega} = (\Omega^c\times\Omega^c)^c$. 
\end{remark}
\begin{remark}
Let us also mention that in the double phase case i.e. when $H$ is of the form
\[
H(x,y,\xi) = \left(\frac{|\xi|}{|x-y|^s}\right)^p+a(x,y)\left(\frac{|\xi|}{|x-y|^r}\right)^q
\]
the condition on initial data ~\cref{datahypo} is satisfied if $u_0\in W^{s,p}(\RR^N)\cap L^2(\Om)$ also satisfies:
\[
\iint\limits_{\mathbb{R}^N\times\mathbb{R}^N}\frac{|u_0(x)-u_0(y)|^p}{|x-y|^{N+sp}}+ a(x,y)\frac{|u_0(x)-u_0(y)|^q}{|x-y|^{N+rq}}\,dx\,dy<\infty.
\]
\end{remark}
\subsection{Main Results}

The following existence and uniqueness theorems are the main results of the paper. 

\begin{theorem}{\textbf{(Existence)}}
    \label{mainthm:e}
    Let $\Om$ be an open bounded subset of $\RR^N$. Suppose that $H$ satisfies the assumptions~\cref{eq:bound_below_H} and~\cref{eq: cvx_H} and let the time-independent Cauchy-Dirichlet data $u_0$ satisfy~\cref{datahypo}. Then,
    \begin{itemize}
        \item there exists a variational solution to~\cref{maineq} in the sense of~\cref{defvar},
        \item the variational solution $u\in C^{0,\frac{1}{2}}([0,T];L^2(\Om))$ for all times $T>0$. Further, $\de_t u\in L^2(\Omega_\infty)$ and 
        \item the following energy bounds are verified
        \[
        \int_0^{\infty}\int_{\Omega}|\de_t u(x,t)|^2 \, dx \, dt \leq \Lambda, \mbox{ and}
         \]  
        \[
         \frac{1}{t_2-t_1}\int_{t_1}^{t_2} \iint\limits_{\RR^N\times\RR^N}\frac{H(x,y,u(x,t)-u(y,t))}{|x-y|^{N}} \,dx \, dy\, dt \leq 2e\Lambda.
         \]
        for $0\leq t_1 < t_2 < \infty$.
    \end{itemize} 
\end{theorem}
 
\begin{theorem}{\textbf{(Uniqueness)}}\label{mainthm:u}
If $\xi\mapsto H(x,y,\xi)$ is strictly convex then there is at most one variational solution in the sense of ~\cref{defvar}.
\end{theorem}

\begin{remark}
    The restriction $p>\frac{2N}{2s+N}$ is often seen as a natural lower bound on $p$ since the proofs of existence that use the monotone operator methods require it on account of the Sobolev embedding. The proof here which uses variational techniques avoids this lower bound to obtain existence in the full range $p\in (1,\infty)$. In the paper \cite{bogeleinExistenceEvolutionaryProblems2019}, functionals with linear growth, that is, corresponding to $p=1$ are also considered by stability methods. It would be of interest to study the stability method for the nonlocal operator.
\end{remark}

\section{Preliminaries}
\subsection{Notations}
We begin by collecting the standard notation that will be used throughout the paper:
\begin{itemize}
\item We shall denote $N$ to be the space dimension. We shall denote by $z=(x,t)$ a point in $ \RR^N\times (0,T)$.  
\item We shall alternately use $\dfrac{\partial f}{\partial t},\partial_t f,f'$ to denote the time derivative of $f$.
 \item Let $\Omega$ be an open bounded domain in $\mathbb{R}^N$ with boundary $\partial \Omega$ and for $0<T\leq\infty$,  let $\Omega_T\coloneqq \Omega\times (0,T)$. 
\item Integration with respect to either space or time only will be denoted by a single integral $\int$ whereas integration on $\Om\times\Om$ or $\RR^N\times\RR^N$ will be denoted by a double integral $\iint$.
\end{itemize}

\subsection{Sobolev spaces}
Let $1<p<\infty$, we denote by $p'=p/(p-1)$ the conjugate exponent of $p$. Let $\Om$ be an open subset of $\RR^N$. We define the {\it Sobolev-Slobodeki\u i} space, which is the fractional analogue of Sobolev spaces.
\begin{align*}
    W^{s,p}(\Om)=\left\{ \psi\in L^p(\Omega): [\psi]_{W^{s,p}(\Om)}<\infty \right\}, s\in (0,1),
\end{align*} where the seminorm $[\cdot]_{W^{s,p}(\Om)}$ is defined by 
\begin{align*}
    [\psi]_{W^{s,p}(\Om)}=\left(\iint\limits_{\Om\times\Om} \frac{|\psi(x)-\psi(y)|^p}{|x-y|^{N+ps}}\,dx\,dy\right)^{\frac 1p}.
\end{align*}
The space when endowed with the norm $\norm{\psi}_{W^{s,p}(\Om)}=\norm{\psi}_{L^p(\Om)}+[\psi]_{W^{s,p}(\Om)}$ becomes a Banach space. The space $W^{s,p}_0(\Om)$ is the subspace of $W^{s,p}(\RR^N)$ consisting of functions that vanish outside $\Om$.

Let $I$ be an interval and let $V$ be a separable, reflexive Banach space, endowed with a norm $\norm{\cdot}_V$. We denote by $V^*$ its topological dual space. Let $v$ be a mapping such that for a.e. $t\in I$, $v(t)\in V$. If the function $t\mapsto \norm{v(t)}_V$ is measurable on $I$, then $v$ is said to belong to the Banach space $L^p(I;V)$ if $\int_I\norm{v(t)}_V^p\,dt<\infty$. It is well known that the dual space $L^p(I;V)^*$ can be characterized as $L^{p'}(I;V^*)$.

\subsection{Mollification in time}
Throughout the paper, we will use the following mollification in time. Let $\Omega$ be an open subset of $\mathbb{R}^N$. For $T>0$, $v\in L^1(\Omega_T)$, $v_0\in L^1(\Omega)$ and $h\in (0,T]$, we define
\begin{align}
    \label{timemollify}
    [v]_h(\cdot,t)=e^{-\frac{t}{h}}v_0 + \frac{1}{h}\int_0^t e^{\frac{s-t}{h}}v(\cdot,s)\,ds,
\end{align} for $t\in [0,T]$.
The convergence properties of mollified functions have been collected in~\ref{Molli}.

\subsection{Auxiliary Results}
We collect the following standard results which will be used in the course of the paper. We begin with a general result on convex minimization.

\begin{proposition}(\cite[Theorem~2.3]{rindlerCalculusVariations2018})\label{mintheorem}
    Let $X$ be a closed affine subset of a reflexive Banach space and let $\mathcal{F}: X\to (-\infty,\infty]$. Assume the following:
    \begin{enumerate}
        \item For all $\Lambda\in\RR$, the sublevel set $\{x\in X: \mathcal{F}[u]<\Lambda\}$ is sequentially weakly precompact, i.e., if for a sequence $(u_j)\subset X$, $\mathcal{F}[u_j]<\Lambda$, then $u_j$ has a weakly convergent subsequence.
        \item For all sequences $(u_j)\subset X$ with $u_j\rightharpoonup u$ in $X$-weak, it holds that
        \begin{align*}
            \mathcal{F}[u] \leq \liminf\limits_{j\to\infty} \mathcal{F}[u_j].
        \end{align*}
    \end{enumerate}
    Then, the minimization problem 
    \begin{align*}
    \mbox{\bf Minimize $\bf\mathcal{F}[u]$ over all $\bf u\in X$}  
    \end{align*}
    has a solution.
\end{proposition}

We will need the following general result on weak lower semicontinuity of functionals.

\begin{proposition}(\cite[Corollary~3.9]{brezisFunctionalAnalysisSobolev2011})
    \label{lowsemcont} 
    Assume that $\phi:E\to (-\infty,\infty]$ is convex and lower semicontinuous in the strong topology. Then $\phi$ is lower semicontinuous in the weak topology.
\end{proposition}

We have the following Sobolev-type inequality~\cite[Theorem~6.5]{dinezzaHitchhikersGuideFractional2012}.

\begin{theorem}[\cite{dinezzaHitchhikersGuideFractional2012}]
    Let $s\in(0,1)$ and $1\leq p<\infty$, $sp<N$ and let $\kappa^*=\frac{N}{N-sp}$, then for any $g\in W^{s,p}(\mathbb{R}^N)$ and $\kappa\in [1,\kappa^*]$, we have
    \begin{align}\label{sobolev1}
        \norm{g}_{L^{\kappa p}}^p\leq C\iint\limits_{\mathbb{R}^N\times\mathbb{R}^N}\frac{|g(x)-g(y)|^p}{|x-y|^{N+sp}}\,dx\,dy.
    \end{align} If $g\in W^{s,p}(\Om)$ and $\Om$ is an extension domain, then
    \begin{align}
        \label{sobolev2}
        ||g||_{L^{\kappa p}(\Om)}^p\leq C\norm{g}_{W^{s,p}(\Om)}.
    \end{align} If $sp=N$, then \cref{sobolev1} holds for all $\kappa\in [1,\infty)$, whereas if $sp>N$, then \cref{sobolev2} holds for all $\kappa\in[1,\infty]$.
\end{theorem}

We will require the compact embedding as follows.

\begin{proposition}(\cite[Proposition~2.1]{hanCompactSobolevSlobodeckijEmbeddings2022})
    \label{compactembed}
    Assume $N\geq 2$, $1\leq p\leq \infty$ and $0<s<1$. Let $\Om$ be a bounded extension domain. When $sp<N$, the embedding $W^{s,p}(\Om)\mapsto L^r(\Om)$ is compact for $r\in [1,p_*)$, and when $sp\geq N$, the same embedding is compact for $r\in [1,\infty)$. The theorem also holds for $W^{s,p}_0(\Om)$ for any bounded domain in $\RR^N$.
\end{proposition}

We also need a Poincar\'e inequality for Gagliardo seminorms.

\begin{proposition}(\cite[Lemma 2.4]{brascoFractionalCheegerProblem2014})
	\label{poincare}
	Let $1\leq p<\infty$, $s\in (0,1)$ and $\Omega$ is an open and bounded set in $\RR^N$. Then, for every $u\in W^{s,p}_0(\Om)$, it holds that
	\begin{align*}
		||u||_{L^p(\Om)}^p\leq C(N,s,p,\Om)\,[u]_{W^{s,p}(\RR^N)}^p
	\end{align*}
\end{proposition}

We end this subsection with the following result about parabolic Banach spaces.
\begin{theorem}(\cite[pp106, Prop.~1.2]{showalterMonotoneOperatorsBanach1997}]\label{integratebypartsintime}
Let the Banach space $V$ be dense and continuously embedded in the Hilbert space $H$; identify $H=H^*$ so that $V\xhookrightarrow{}H\xhookrightarrow{} V^*$. The Banach space $W_p(0,T)=\{u\in L^p(0,T;V):\de_t v\in L^p(0,T;V^*)\}$ is contained in $C([0,T];H)$. Moreover, if $u\in W_p(0,T)$, then $|u(\cdot)|^2_H$ is absolutely continuous on $[0,T]$,
\begin{align*}
    \frac{d}{dt}|u(t)|^2_H=2\de_t u(t)u(t) \mbox{ a.e. } t\in [0.T],
\end{align*} and there is a constant $C$ such that
\begin{align*}
    \norm{u}_{C([0,T];H)}\leq C\norm{u}_{W_p(0,T)}, u\in W_p.
\end{align*}  
Moreover, if $u,v\in W_p(0,T)$, then $\iprod{u(\cdot)}{v(\cdot)}_H$ is absolutely continuous on $[0,T]$ and
\begin{align*}
    \frac{d}{dt}\iprod{u(t)}{v(t)}_H=\de_t u(t)v(t)+\de_t v(t) u(t), \mbox{ a.e. }t\in[0,T].
\end{align*}
\end{theorem}

\section{Existence of variational solutions}
In this section, we will prove the existence of variational solutions to the following nonlocal evolution equation with time independent initial-boundary data
\begin{equation*}
    \left\{
		\begin{array}{lll}
			\de_t u + Lu = 0 & \mbox{ in } \Omega_{\infty} \\
			u = u_0 &\text{ on } (\mathbb{R}^N\setminus\Omega)\times(0,\infty)\cup {\Omega}\times\{0\}
		\end{array}
	\right.
\end{equation*}
where $\Omega$ is an open bounded subset of $\bb{R}^N$ and $\Omega_{\infty} = \Omega\times (0,\infty)$. We recall that the operator $L$ is a nonlocal operator whose explicit structure is as follows.
\begin{equation*}
    Lu(x,t) = P.V.\int_{\mathbb{R}^N} \frac{D_{\xi}H(x,y,u(x,t)-u(y,t))}{|x-y|^N} \,dy
\end{equation*}
where $H=H(x,y,\xi) : \bb{R}^N \times \bb{R}^N \times \bb{R} \rightarrow \bb{R}$ is a measurable function satisfying~\cref{eq:bound_below_H} and~\cref{eq: cvx_H}.
 We recall the assumptions on initial-boundary data. The initial-boundary $u_0 : \bb{R}^N \rightarrow \bb{R}$ is assumed to be time independent and $u_0 \in W^{s,p}(\RR^N) \cap L^2(\Omega)$. We further assume that 
\begin{equation}\label{eq:initial}
    \iint\limits_{\RR^N\times\RR^N}\frac{H(x,y,u_0(x)-u_0(y))}{|x-y|^N} \,dx \,dy \leq \Lambda < \infty  
\end{equation}
for some constant $\Lambda > 0$.\\
Finally, we recall our definition of solution to the equation $\eqref{maineq}$. We denote by $W_{(u_0)}^{s,p}(\Omega)$ the following space of functions  
\[
W_{(u_0)}^{s,p}(\Omega) = \{ g : \bb{R}^N \rightarrow \bb{R} : g \in W^{s,p}(\bb{R}^N), \,\, g(x) = u_0(x) \, \forall x \in \Omega^c\}  
\]
\begin{definition}
We say that a function 
\[u \in L^p(0,T;W_{(u_0)}^{s,p}(\Omega)) \cap C([0,T];L^2(\Omega))\] 
for all $T>0$ is a variational solution to $\eqref{maineq}$ if for all $T>0$ and any comparison function 
\[v \in L^p(0,T;W_{(u_0)}^{s,p}(\Omega))\] with 
\[\de_t v \in L^2(\Omega\times(0,T))\] 
the following variational inequality is verified:
\begin{equation}\label{eq:vareq}
\begin{split}
    \int_0^T \iint\limits_{\RR^N\times\RR^N}\frac{H(x,y,u(x,t)-u(y,t))}{|x-y|^N} \,dx \,dy \,dt \leq \int_0^T \iint\limits_{\RR^N\times\RR^N}\frac{H(x,y,v(x,t)-v(y,t))}{{|x-y|^N}} \,dx \,dy \,dt \\ + \int_0^T \int_{\Omega}\de_tv\cdot(v-u) + \frac{1}{2}\norm{v(\cdot,0) - u_0(\cdot)}_{L^2(\Omega)}^2 - \frac{1}{2}\norm{v(\cdot,T) - u(\cdot,T)}_{L^2(\Omega)}^2
\end{split}
\end{equation}
\end{definition}
\begin{remark}\label{par:initial_data}(Initial Data)
Since the initial data $u_0$ is a valid comparison function and $H \geq 0$,  The inequality $\eqref{eq:vareq}$ with the admissible comparision function $v(\cdot,t)=u_0(\cdot)$ for all $t\in(0,\infty)$ and $\eqref{eq:initial}$ implies that for any $T>0$
\[
0 \leq \frac{1}{2}\norm{u(\cdot,T) - u_0(\cdot)}_{L^2(\Omega)}^2 \leq T\Lambda \rightarrow 0 \text{ as } T \rightarrow 0+.
\]
So variational solutions pick up the data at the initial time $t=0$ in the $L^2$ sense. 
\end{remark}


\subsection{Lower bound for the functional} 
In this subsection, we derive a lower bound on the nonlinear and nonlocal operator by the use of Poincar\'e inequality (\cite[Lemma 2.4]{brascoFractionalCheegerProblem2014}).

Note that for a fixed time $t\geq 0$, $(u-u_0)(\cdot,t) \in W^{s,p}_0(\Omega)$. Therefore, by Poincar\'e inequality (~\cref{poincare}) in the space $W_0^{s,p}(\Om)$, we get for any $t\in [0,T]$,
\begin{align*}
	||u-u_0||_{L^p(\Om)}^p&\leq C [u-u_0]^p_{W^{s,p}(\RR^N)}\\
	&\leq C [u]_{W^{s,p}(\RR^N)}^p + C [u_0]_{W^{s,p}(\RR^N)}^p\\
	&\leq C \iint\limits_{\RR^N\times\RR^N} \frac{H(x,y,u(x,t)-u(y,t))}{|x-y|^N}\,dx\,dy+ C [u_0]_{W^{s,p}(\RR^N)}^p,
\end{align*} where we made use of $\eqref{eq:bound_below_H}$. As a result, we obtain
\begin{align}
	||u||_{L^p(\RR^N)}^p &\leq ||u-u_0||_{L^p(\Om)}^p+||u_0||_{L^p(\RR^N)}^p\nonumber\\
	&\leq C_1 \iint\limits_{\RR^N\times\RR^N} \frac{H(x,y,u(x,t)-u(y,t))}{|x-y|^N}\,dx\,dy+ C_2 ||u_0||_{W^{s,p}(\RR^N)}^p,
\end{align} for some positive constants $C_1, C_2$. Once again, due to \cref{eq:bound_below_H}, we get
\begin{align}\label{eq:lowerboundint_H}
	||u||_{W^{s,p}(\RR^N)}^p \leq C_1 \iint\limits_{\RR^N\times\RR^N} \frac{H(x,y,u(x,t)-u(y,t))}{|x-y|^N}\,dx\,dy+ C_2 ||u_0||_{W^{s,p}(\RR^N)}^p.
\end{align}

\subsection{Convex Minimization} For $\ve \in (0,1]$ we define the function space $K_{\ve}$ as follows
\[
K_{\ve} := \left\{g:\bb{R}^N\times(0,\infty) \rightarrow \bb{R} : g \in W^{s,1}(\bb{R}^N \times (0,T)) \, \forall T>0, \, \text{ and } \norm{g}_{\ve} < \infty \right\}
\]
where 
\[
\norm{g}_{\ve} := \left(\int_0^{\infty}\int_{\Omega}e^{-\frac{t}{\ve}}|\de_t g|^2 \, dx \, dt \right)^{1/2} + \left(\int_0^{\infty}\int_{\RR^N}e^{-\frac{t}{\ve}}|g|^p + \int_0^{\infty} \iint\limits_{\RR^N\times\RR^N}e^{-\frac{t}{\ve}}\frac{|g(x,t)-g(y,t)|^p}{|x-y|^{N+sp}} \,dx \, dy\, dt \right)^{1/p}.
\]
Let $N_{\ve}$ denote the subspace of $K_{\ve}$ which contains functions $v:\bb{R}^n \rightarrow \bb{R}$ such that $v(x,t) = 0$ for every $(x,t) \in \Omega^c \times (0,\infty)$. Then both $K_{\ve}$ and $N_{\ve}$ are Banach spaces under the norm $\norm{\cdot}_{\ve}$. In particular, $u_0 + N_{\ve}$ is a (weakly) closed subset of $K_{\ve}$; where by $g \in u_0 + N_{\ve}$ we mean functions $g$ such that $g-u_0 \in N_{\ve}$. \\
For $\ve \in (0,1]$ consider the following variational integral 
\[
F_{\ve}(v) := \frac{1}{2}\int_0^{\infty}\int_{\Omega}e^{-\frac{t}{\ve}}|\de_t v|^2 \, dx \, dt + \frac{1}{\ve} \int_0^{\infty} \iint\limits_{\RR^N\times\RR^N}e^{-\frac{t}{\ve}}\frac{H(x,y,v(x,t)-v(y,t))}{|x-y|^{N}} \,dx \, dy\, dt.
\]
We note that $F_{\ve} : K_{\ve} \rightarrow (-\infty,\infty]$ satisfies the following properties on the closed subset $u_0 + N_{\ve}$
\begin{itemize}
    \item {\bf Admissibile set is nonempty:} The admissible set $(u_0 + N_{\ve})\cap\{F_{\ve}(v) < \infty\}$ is non-empty because the time independent extension $U(\cdot,t) = u_0(\cdot)$ for all $t>0$ is a function in the intersection due to the assumptions on $u_0$ (see $\ref{eq:initial})$, 
    \item {\bf Coercivity of $F_{\ve}$:} This is an easy consequence of \cref{eq:lowerboundint_H}. Multiplying both sides of \cref{eq:lowerboundint_H} by $e^{-\frac{t}{\ve}}$ and integrating over $(t_1,\infty)$ for some $t_1\in [0,\infty)$, we obtain
    \begin{align*}
    	\int_{t_1}^\infty& e^{-\frac{t}{\ve}}\iint\limits_{\RR^N\times\RR^N} \frac{|u(x,t)-u(y,t)|^p}{|x-y|^{N+ps}}\,dx\,dy\,dt+
    	\int_{t_1}^\infty e^{-\frac{t}{\ve}}\int\limits_{\RR^N} |u(x,t)|^p\,dx\,dt\leq\\
    	&C_1\int_{t_1}^\infty e^{-\frac{t}{\ve}}\iint\limits_{\RR^N\times\RR^N} \frac{H(x,y,u(x,t)-u(y,t))}{|x-y|^{N}}\,dx\,dy\,dt+C_2\ve e^{-\frac{t_1}{\ve}}||u_0||^p_{W^{s,p}(\RR^N)}.
    \end{align*} 
	Choosing $t_1=0$ and adding the remaining term from the norm of $K_\ve$, we obtain
	\begin{align*}
		F_\ve(u)\geq \frac{1}{C_1\ve}\min \{ ||u||_{K_\ve}^p, ||u||_{K_\ve}^2 \}-\frac{C_2}{C_1}||u_0||_{W^{s,p}(\RR)^N}^p.
	\end{align*}
    \item {\bf Strict convexity of $F_{\ve}$:} $ 
    F_{\ve}$ is strictly convex due to convexity of $H$ in $\xi$ and the strict convexity of $Z \mapsto |Z|^2$. 
    \item {\bf Weak lower-semicontinuity of $F_{\ve}$:} In order to show weak lower-semicontinuity of $F_{\ve}$, we will make use of \cref{lowsemcont}, that is, we will show that $F_{\ve}$ is strongly lower semicontinuous and convex. Convexity of $F_{\ve}$ was shown in the previous step. It remains to show that $F_{\ve}$ is strongly lower semicontinuous. 
    
    We begin by writing 
    \begin{align*}
        F_{\ve}(v) &= \frac{1}{2}\int_0^{\infty}\int_{\Omega}e^{-\frac{t}{\ve}}|\de_t v|^2 \, dx \, dt + \frac{1}{\ve} \int_0^{\infty} \iint\limits_{\RR^N\times\RR^N}e^{-\frac{t}{\ve}}\frac{H(x,y,v(x,t)-v(y,t))}{|x-y|^{N}} \,dx \, dy\, dt\\
        & := F_1(v) + F_2(v).
    \end{align*}

	Let $u_j$ be strongly convergent to $u$ in $K_\ve$. Then $w_j:= e^{-\frac{t}{\ve}}|\de_t u_j|^2$ converges strongly in $L^1(0,\infty;\Omega)$. Therefore, we obtain 
    \begin{align}\label{liminf1}F_1(u)=\lim\limits_{j\to\infty} F_1(u_j)= \liminf\limits_{j\to\infty} F_1(u_j).\end{align} 
    
    On the other hand, $u_j$ converging strongly to $u$ in $K_\ve$ also implies that $e^{-\frac{t}{p\ve}}u_j$ converges strongly to $e^{-\frac{t}{p\ve}}u$ in $L^p(0,T;L^{p}(\RR^N))$. As a result, for a subsequence, $u_j(t,x)\to u(t,x)$ for a.e. $t\in (0,\infty)$ and $x\in\RR^N$. By continuity of $H$ in the last variable, we have
    \begin{align*}
        \frac{H(x,y,u_j(x,t)-u_j(x,t))}{|x-y|^N}\to\frac{H(x,y,u(x,t)-u(y,t))}{|x-y|^N}\,a.e.\,x,y\in\RR^N, t\in (0,\infty).
    \end{align*}
    Finally, by Fatou's lemma,
    \begin{align}
        F_2(u)&=\frac{1}{\ve}\int_0^\infty\iint\limits_{\RR^N\times\RR^N}e^{-\frac{t}{\ve}}\frac{H(x,y,u(x,t)-u(y,t))}{|x-y|^N}\,dx\,dy\,dt\nonumber\\
        &\leq  \frac{1}{\ve}\liminf\limits_{j\to\infty}\int_0^\infty\iint\limits_{\RR^N\times\RR^N}e^{-\frac{t}{\ve}}\frac{H(x,y,u_j(x,t)-u_j(x,t))}{|x-y|^N}\,dx\,dy\,dt=\liminf\limits_{j\to\infty}F_2(u_j).\label{liminf2}
    \end{align}
    Combining \cref{liminf1} and \cref{liminf2}, we get 
    \begin{align*}
        F_\ve(u) & = F_1(v) + F_2(v)\\
                 & \leq \liminf\limits_{j\to\infty} F_1(u_j)+\liminf\limits_{j\to\infty} F_2(u_j)\\
                 & \leq \liminf\limits_{j\to\infty} F_\ve(u_j),
    \end{align*} where in the last inequality, we have used the super-additivity of $\liminf$.
\end{itemize} 
It follows that for any $\ve \in (0,1]$, by \cref{mintheorem} $F_{\ve}$ admits a unique minimizing map $u_{\ve} \in u_0 + N_{\ve}$. We record this as a lemma below.

\begin{lemma}
	For any $\ve\in (0,1]$, the functional $F_{\ve}$ admits a unique minimizer $u_\ve\in (u_0+N_\ve)\cap \{u:F_\ve(u)<\infty\}$.
\end{lemma}

\subsection{Energy Estimates} We note that since the time independent extension $U(\cdot,t) = u_0(\cdot)$ is a valid comparison function in $u_0 + N_{\ve}$ we have 
\[
F_{\ve}(u_{\ve}) \leq F_{\ve}(U) = \frac{1}{\ve} \int_0^{\infty} \iint\limits_{\RR^N\times\RR^N}e^{-\frac{t}{\ve}}\frac{H(x,y,u_0(x)-u_0(y))}{|x-y|^{N}} \,dx \, dy\, dt
\]
which using $\eqref{eq:initial}$ becomes 
\begin{equation}\label{eq:lem1}
    F_{\ve}(u_{\ve}) \leq \Lambda. 
\end{equation}
We obtain the following $\ve$ dependent energy bounds directly from the definition of a minimizer and $\eqref{eq:lem1}$. 
\begin{align*}
&\int_0^{\infty}\int_{\Omega}e^{-\frac{t}{\ve}}|\de_t u_{\ve}|^2 \, dx \, dt \leq 2\Lambda \\
&\int_0^{\infty} \iint\limits_{\RR^N\times\RR^N}e^{-\frac{t}{\ve}}\frac{H(x,y,u_{\ve}(x,t)-u_{\ve}(y,t))}{|x-y|^{N}} \,dx \, dy\, dt \leq \ve \Lambda. 
\end{align*}
\subsection{Improvement} \label{par:uniformest} We now want to improve our energy estimates to make it independent of $\ve$. To this end, we define the following additional functions for $t>0$.
\begin{itemize}
    \item $L_{\ve}(t) := \frac{1}{2}\int_{\Omega}|\de_t u_{\ve}(x,t)|^2 \,dx$ 
    \item $G_{\ve}(t) := \frac{1}{2}\int_{\Omega}|\de_t u_{\ve}(x,t)|^2 \,dx + \frac{1}{\ve}\iint\limits_{\RR^N\times\RR^N}\frac{H(x,y,u_{\ve}(x,t)-u_{\ve}(y,t))}{|x-y|^{N}} \,dx \, dy$
    \item $I_{\ve}(t) := \int_t^{\infty}e^{-\frac{s}{\ve}}G_{\ve}(s) \,ds$. 
\end{itemize}
We note that all these additional functions are nonnegative. Furthermore, $I_{\ve}(t)$ decreases from $F_{\ve}(u_{\ve})$ to $0$ as $t \rightarrow \infty$. Our next lemma says that $I_{\ve}$ continues to be a decreasing function even when multiplied by an exponential function. This will help us get a uniform bound on the time derivatives.  
\begin{lemma}\label{lem:lem1}
For a.e. $t \in (0,\infty)$, 
\[
\frac{d}{dt}\left(e^{t/\ve}I_{\ve}(t)\right) = -2L_{\ve}(t) \leq 0.
\]
In particular, $e^{t/\ve}I_{\ve}(t)$ is a decreasing function in $t$. 
\end{lemma}
\begin{proof}
Let $\xi \in C_0^{\infty}(0,\infty)$ and for $\delta \in \bb{R}$ define $\phi_{\delta} \in C^{\infty}(0,\infty)$ by 
\[
\phi_{\delta}(s) := s + \delta\xi(s).
\]
Then for $|\delta| << 1$, $\phi_{\delta}$ is a diffeomorphism from $(0,\infty)$ onto itself. Let $\psi_{\delta}$ denote the inverse of $\phi_{\delta}$. We define the inner variation of our minimizer by 
\[
u_{\ve}^{\delta}(x,s) : = u_{\ve}(x,\phi_{\delta}(s))
\]
and compute
\begin{align*}
    F_{\ve}(u_{\ve}^{\delta}) &= \frac{1}{2}\int_0^{\infty}\int_{\Omega}e^{\frac{-s}{\ve}}|\de_t u_{\ve}(x,\phi_{\delta}(s))|^2 \phi_{\delta}'(s)^2 \, dx \, ds \\
    &\qquad+
\frac{1}{\ve}\int_0^{\infty} \iint\limits_{\RR^N\times\RR^N}e^{\frac{-s}{\ve}}\frac{H(x,y,u_{\ve}(x,\phi_{\delta}(s))-u_{\ve}(y,\phi_{\delta}(s)))}{|x-y|^{N}} \,dx \, dy\, ds \\
&= \frac{1}{2}\int_0^{\infty}\int_{\Omega}e^{\frac{-\psi_{\delta}(t)}{\ve}}|\de_t u_{\ve}(x,t)|^2\phi_{\delta}'(\psi_{\delta}(t))^2\psi_{\delta}'(t) \, dx \, dt \\
&\qquad+
\frac{1}{\ve}\int_0^{\infty} \iint\limits_{\RR^N\times\RR^N}e^{\frac{-\psi_{\delta}(t)}{\ve}}\frac{H(x,y,u_{\ve}(x,t)-u_{\ve}(y,t))}{|x-y|^{N}}\psi_{\delta}'(t) \,dx \, dy\, dt \\
&= \frac{1}{2}\int_0^{\infty}\int_{\Omega}e^{\frac{-\psi_{\delta}(t)}{\ve}}\frac{1}{\psi_{\delta}'(t)}|\de_t u_{\ve}(x,t)|^2 \, dx \, dt \\
&\qquad+
\frac{1}{\ve}\int_0^{\infty} \iint\limits_{\RR^N\times\RR^N}e^{\frac{-\psi_{\delta}(t)}{\ve}}\frac{H(x,y,u_{\ve}(x,t)-u_{\ve}(y,t))}{|x-y|^{N}}\psi_{\delta}'(t) \,dx \, dy\, dt
\end{align*}
Thus for $|\delta| << 1$ we have  \[F_{\ve}(u_{\ve}^{\delta}) < \infty.\]
Next, we note that at $\delta = 0$, $u_{\ve}^{\delta} = u_{\ve}$ and recalling that $u_{\ve}$ minimizes $F_{\ve}$, it follows that $\delta \mapsto F_{\ve}(u_{\ve}^{\delta})$ admits a minimizer at $\delta = 0$. Therefore, 
\begin{align*}
0 = \frac{d F_{\ve}(u^\delta_{\ve})}{d\delta}|_{\delta = 0}
&= \frac{1}{2\ve}\int_0^{\infty}\int_{\Omega}e^{-\frac{t}{\ve}}\xi(t)|\de_t u_{\ve}(x,t)|^2  \, dx \, dt \\
&\qquad+
\frac{1}{\ve^2}\int_0^{\infty} \iint\limits_{\RR^N\times\RR^N}e^{-\frac{t}{\ve}}\xi(t)\frac{H(x,y,u_{\ve}(x,t)-u_{\ve}(y,t))}{|x-y|^{N}} \,dx \, dy\, dt \\
&\qquad\qquad+ \frac{1}{2}\int_0^{\infty}\int_{\Omega}e^{-\frac{t}{\ve}}\xi'(t)|\de_t u_{\ve}(x,t)|^2 \, dx \, dt\\
&\qquad\qquad\qquad -
\frac{1}{\ve}\int_0^{\infty} \iint\limits_{\RR^N\times\RR^N}e^{-\frac{t}{\ve}}\xi'(t)\frac{H(x,y,u_{\ve}(x,t)-u_{\ve}(y,t))}{|x-y|^{N}} \,dx \, dy\, ds \\
&= \int_0^{\infty}\left(-\xi(t)\frac{1}{\ve}I_{\ve}'(t)+\xi'(t)I_{\ve}'(t)+\xi'(t)2e^{-\frac{t}{\ve}}L_{\ve}(t)\right)\,dt\\
&= \int_0^{\infty}\left(\frac{1}{\ve}I_{\ve}(t)+I_{\ve}'(t)+2e^{-\frac{t}{\ve}}L_{\ve}(t)\right)\xi'(t)\,dt
\end{align*}
Since $\xi \in C_0^{\infty}(0,\infty)$ was arbitrary, it follows that 
\[
\frac{1}{\ve}I_{\ve}(t)+I_{\ve}'(t)+2e^{-\frac{t}{\ve}}L_{\ve}(t)
\]
is a nonnegative constant for a.e. $t \in (0,\infty)$. On the other hand, $I_{\ve}'(t)+2e^{-\frac{t}{\ve}}L_{\ve}(t)$ is integrable on $(0,\infty)$ and $I_{\ve}(t) \rightarrow 0$ as $t \rightarrow \infty$; so the constant must be zero. The conclusion follows.  
\end{proof}
\begin{remark}
We note that since $I_{\ve}(0) = F_{\ve}(u_{\ve})$, the above lemma along with $\eqref{eq:lem1}$ implies the following bound
\begin{equation}\label{eq:expIepsbound}
    e^{\frac{t}{\ve}}I_{\ve}(t) \leq \Lambda. 
\end{equation}
\end{remark}
\begin{lemma}\label{lem:uniformtimebound}
The minimizer $u_{\ve}$ satisfies the following uniform bound
\[
\int_0^{\infty}\int_{\Omega}|\de_t u_{\ve}|^2 \, dx \, dt \leq \Lambda.
\]
\end{lemma}
\begin{proof}
For $0 < t_1 < t_2 < \infty$ we compute
\[
2\int_{t_1}^{t_2}L_{\ve}(t) \,dt = e^{\frac{t_1}{\ve}}I_{\ve}(t_1) - e^{\frac{t_2}{\ve}}I_{\ve}(t_2) \leq \Lambda
\]
where we used $\eqref{eq:expIepsbound}$. Letting $t_1 \rightarrow 0+$ and $t_2 \rightarrow \infty$, the conclusion follows. 
\end{proof}
\begin{lemma}\label{lem:spacebound}
The minimizer $u_{\ve}$ satisfies the following bound for any $0 \leq t_1 < t_2$ with $t_2-t_1 \geq \ve$
\[
\frac{1}{t_2-t_1}\int_{t_1}^{t_2} \iint\limits_{\RR^N\times\RR^N}\frac{H(x,y,u_{\ve}(x,t)-u_{\ve}(y,t))}{|x-y|^{N}} \,dx \, dy\, dt \leq 2e\Lambda.
\]
\end{lemma}
\begin{proof}
For $0<t<\infty$ and $0 < \delta \leq \ve$ we compute
\begin{align*}
    \int_{t}^{t+\delta} \iint\limits_{\RR^N\times\RR^N}&\frac{H(x,y,u_{\ve}(x,s)-u_{\ve}(y,s))}{|x-y|^{N}} \,dx \, dy\, ds \\
    &\leq e^{\frac{t+\ve}{\ve}}\int_{t}^{t+\delta} e^{\frac{-s}{\ve}} \iint\limits_{\RR^N\times\RR^N}\frac{H(x,y,u_{\ve}(x,t)-u_{\ve}(y,t))}{|x-y|^{N}} \,dx \, dy\, ds \\
    &\leq e^{\frac{t+\ve}{\ve}}\int_{t}^{\infty} e^{\frac{-s}{\ve}} \iint\limits_{\RR^N\times\RR^N}\frac{H(x,y,u_{\ve}(x,t)-u_{\ve}(y,t))}{|x-y|^{N}} \,dx \, dy\, ds \\
    &\leq \ve e e^{\frac{t}{\ve}}I_{\ve}(t) \\
    &\leq \ve e \Lambda
\end{align*}
where we used $\eqref{eq:expIepsbound}$. Choose a natural number $j$ such that $(j-1)\ve < t_2-t_1 \leq j \ve$. Then $j\ve \leq t_2-t_1+\ve \leq 2(t_2-t_1)$ and so
\begin{align*}
    \int_{t_1}^{t_2} \iint\limits_{\RR^N\times\RR^N}\frac{H(x,y,u_{\ve}(x,t)-u_{\ve}(y,t))}{|x-y|^{N}} \,dx \, dy\, dt
    \leq  \int_{t_1+(j-1)\ve}^{t_2}\iint\limits_{\RR^N\times\RR^N}\frac{H(x,y,u_{\ve}(x,t)-u_{\ve}(y,t))}{|x-y|^{N}} \,dx \, dy\, dt \\ + \sum_{i=0}^{j-2} \int_{t_1+i\ve}^{t_1+(i+1)\ve}\iint\limits_{\RR^N\times\RR^N}\frac{H(x,y,u_{\ve}(x,t)-u_{\ve}(y,t))}{|x-y|^{N}} \,dx \, dy\, dt\\
    \leq j\ve e\lambda 
    \leq 2(t_2-t_1)e\Lambda.
\end{align*}
The conclusion follows. 
\end{proof}
\begin{corollary}\label{cor:uniformbound}
The minimizers $u_{\ve}$ satisfy the following uniform energy estimate
\[
C_1\int_0^T\norm{u_{\ve}(\cdot,t)}_{W^{s,p}}^p \, dt  \leq 2eT\Lambda + TC_2\norm{u_0}_{s,p}^p
\]
for any $T \geq \ve$ and $C_1,C_2$ are positive constants as in $\eqref{eq:lowerboundint_H}$.
\end{corollary}
\begin{proof}
Lemma $\ref{lem:spacebound}$ implies that 
\[
\int_0^{T} \iint\limits_{\RR^N\times\RR^N}\frac{H(x,y,u_{\ve}(x,t)-u_{\ve}(y,t))}{|x-y|^{N}} \,dx \, dy\, dt \leq 2eT\Lambda.
\]
On the other hand, integrating $\eqref{eq:lowerboundint_H}$ from $0$ to $T$ gives
\[
\int_0^{T}\iint\limits_{\RR^N\times\RR^N}\frac{H(x,y,v(x,t)-v(y,t))}{|x-y|^N} \,dx \,dy \,dt \geq C_1\int_0^{T}\norm{v(\cdot,t)}_{W^{s,p}}^p\,dt - TC_2\norm{u_0}_{s,p}^p. 
\]
Combining the two inequalities above, the conclusion follows. 
\end{proof}
\subsection{Weak limit of minimizers} Lemma $\ref{lem:uniformtimebound}$ and Corollary $\ref{cor:uniformbound}$ imply that up to a subsequence there is a function $u: \bb{R}^N \rightarrow \bb{R}$ such that for every $T>0$ $u \in L^p(0,T;W_{(u_0)}^{s,p}(\Omega))$, $\de_t u \in L^2(\Omega_{\infty})$ and we have the following convergences
\begin{itemize}
    \item $u_{\ve} \rightharpoonup u$ in $L^p((0,T);W^{s,p}_{(u_0)}(\Om))$,
    \item $\de_t u_{\ve} \rightharpoonup \de_t u$ in $L^2(\Omega_{\infty})$.
\end{itemize}
By lower semi-continuity with respect to weak convergence of $H$ in $\xi$ and of the $L^2$ norm, Lemma $\ref{lem:uniformtimebound}$ and Lemma $\ref{lem:spacebound}$ imply
\[
\int_0^{\infty}\int_{\Omega}|\de_t u(x,t)|^2 \, dx \, dt \leq \Lambda, \mbox{ and}
\]
\[
\frac{1}{t_2-t_1}\int_{t_1}^{t_2} \iint\limits_{\RR^N\times\RR^N}\frac{H(x,y,u(x,t)-u(y,t))}{|x-y|^{N}} \,dx \, dy\, dt \leq 2e\Lambda.
\]
for $0\leq t_1 < t_2 < \infty$. In particular, the first estimate implies that
\begin{align}\label{timeregular}
    \norm{u(\cdot,t_2)-u(\cdot,t_1)}_{L^2(\Omega)}^2 = \int_{\Omega}\left|\int_{t_1}^{t_2} \de_t u( \cdot, t) \, dt\right|^2 \, dx 
    \leq |t_2-t_1|\int_{t_1}^{t_2}\int_{\Omega}|\de_t u|^2 \,dx \, dt 
    \leq \Lambda|t_2-t_1|.
\end{align}
\begin{remark}\label{rem:yadayada}
The calculation~\cref{timeregular} also holds for $u$ replaced with $u_{\ve}$. 
\end{remark}
We take $t_1 = 0$ to get
\begin{align*}
    \int_{\Omega}|\de_t u(\cdot,t)|^2 \,dx \leq  2 \int_{\Omega} |u_0|^2 \, dx + 2t\Lambda.
\end{align*}
It follows that for any $T>0$ we have
\[
u \in C^{0,\frac{1}{2}}([0,T];L^2(\Omega)).
\]
\subsection{Proof of Existence}

In this subsection, we will show that $u$ satisfies the variational inequality $\eqref{eq:vareq}$. The computation is similar to~\cite[Subsection~4.3]{bogeleinExistenceEvolutionaryVariational2014}. 

Fix a time $T>0$. For $\theta \in (0,T/2)$ we define
\[
\xi_{\theta}(t) := \left\{
		\begin{array}{lll}
			\frac{1}{\theta}t & \mbox{ if } t \in  [0,\theta)\\
			1 &\text{ if } t \in  [\theta,T-\theta]\\
		\frac{1}{\theta}(T-t) & \mbox{ if } t \in  (T-\theta,T]\\
		\end{array}
	\right.
\]
For any $\phi \in L^p(0,T;W^{s,p}_0(\Omega))$ with $\de_t \phi \in L^2(\Omega_T)$ and any $0 < \ve, \delta < 1$ the function 
\[
\psi = \psi_{\ve,\delta}(\cdot,t) = \delta e^{\frac{t}{\ve}}\phi(\cdot,t)1_{[0,T]}(t)
\]
is contained in $N_{\ve}$. So, $u_{\ve} + \psi$ is a valid comparison function for the minimization problem for $F_{\ve}$ and
\[
F_{\ve}(u_{\ve}) \leq F(u_{\ve}+\psi)
\]
which, using the convexity of $H$, implies for $\delta$ small that 
\[
\scalemath{0.85}{
\begin{aligned}
    0 &\leq \frac{1}{2}\int_0^{T}\int_{\Omega}e^{-\frac{t}{\ve}}\left(|\de_t u_{\ve}+\delta\de_t(e^{\frac{t}{\ve}}\xi_{\theta}\phi)|^2 - |\de_t u_{\ve}|^2 \right)\, dx \, dt \\ 
    &\qquad+
\frac{1}{\ve}\int_0^{T} \iint\limits_{\RR^N\times\RR^N}e^{-\frac{t}{\ve}}\left(\frac{H(x,y,(u_{\ve}+\delta e^{\frac{t}{\ve}}\xi_{\theta}\phi)(x,t)-(u_{\ve}+\delta e^{\frac{t}{\ve}}\xi_{\theta}\phi)(y,t))}{|x-y|^{N}} - \frac{H(x,y,u_{\ve}(x,t)-u_{\ve}(y,t))}{|x-y|^{N}}\right) \,dx \, dy\, dt
\\
&\leq 
\int_0^{T}\int_{\Omega}e^{-\frac{t}{\ve}}\left(\frac{1}{2}\delta^2|\de_t(e^{\frac{t}{\ve}}\xi_{\theta}\phi)|^2 + \delta\de_tu_{\ve}\cdot\de_t(e^{\frac{t}{\ve}}\xi_{\theta}\phi) \right)\, dx \, dt \\ 
&
\qquad+\frac{1}{\ve}\int_0^{T} \iint\limits_{\RR^N\times\RR^N}e^{-\frac{t}{\ve}}\delta e^{\frac{t}{\ve}}\xi_{\theta}\left(\frac{H(x,y,(u_{\ve}+\phi)(x,t)-(u_{\ve}+\phi)(y,t))}{|x-y|^{N}} - \frac{H(x,y,u_{\ve}(x,t)-u_{\ve}(y,t))}{|x-y|^{N}}\right) \,dx \, dy\, dt
\end{aligned}
}
\]
Multiplying throughout by $\ve/\delta$ and letting $\delta \rightarrow 0+$ we get 
\[
\scalemath{0.85}{
\begin{aligned}
    0 &\leq \int_0^{T}\int_{\Omega}e^{-\frac{t}{\ve}} \ve \de_tu_{\ve}\cdot\de_t(e^{\frac{t}{\ve}}\xi_{\theta}\phi) \, dx \, dt \\ 
    &+
\int_0^{T} \iint\limits_{\RR^N\times\RR^N}e^{-\frac{t}{\ve}}e^{\frac{t}{\ve}}\xi_{\theta}\left(\frac{H(x,y,(u_{\ve}+\phi)(x,t)-(u_{\ve}+\phi)(y,t))}{|x-y|^{N}} - \frac{H(x,y,u_{\ve}(x,t)-u_{\ve}(y,t))}{|x-y|^{N}}\right) \,dx \, dy\, dt
\\
&= 
\int_0^{T} \iint\limits_{\RR^N\times\RR^N}\xi_{\theta}\left(\frac{H(x,y,(u_{\ve}+\phi)(x,t)-(u_{\ve}+\phi)(y,t))}{|x-y|^{N}} - \frac{H(x,y,u_{\ve}(x,t)-u_{\ve}(y,t))}{|x-y|^{N}}\right) \,dx \, dy\, dt \\
&\qquad+ \int_0^{T}\int_{\Omega} \xi_{\theta}\de_t u_{\ve}\cdot\phi\,dx\,dt+ \ve\int_0^{T}\int_{\Omega} \left(\xi_{\theta}'\de_tu_{\ve}\cdot\phi+ \xi_{\theta}\de_tu_{\ve}\cdot\de_t\phi\right) \, dx \, dt 
\end{aligned}
}
\]
Let $v \in L^p(0,T;W_{(u_0)}^{s,p}(\Omega))$ with 
$\de_t v \in L^2(\Omega\times(0,T))$ and put $\phi = v - u_{\ve}$ above to get
\begin{align}\label{PQRS}
  \int_0^{T} \iint\limits_{\RR^N\times\RR^N} \frac{H(x,y,u_{\ve}(x,t)-u_{\ve}(y,t))}{|x-y|^{N}} \,dx \, dy\, dt \leq \underbrace{\int_0^{T} \iint\limits_{\RR^N\times\RR^N} (1-\xi_{\theta})\frac{H(x,y,u_{\ve}(x,t)-u_{\ve}(y,t))}{|x-y|^{N}} \,dx \, dy\, dt}_P \nonumber\\
  + \underbrace{\int_0^{T}\int_{\Omega} \xi_{\theta}\de_t u_{\ve}\cdot(v-u_{\ve}) \,dx\,dt}_Q + \underbrace{\int_0^{T} \iint\limits_{\RR^N\times\RR^N}\xi_{\theta}\frac{H(x,y,v(x,t)-v(y,t))}{|x-y|^{N}}\,dx \, dy\, dt}_R \nonumber\\
  + \underbrace{\ve\int_0^{T}\int_{\Omega} \left(\xi_{\theta}'\de_tu_{\ve}\cdot(v-u_{\ve})+ \xi_{\theta}\de_tu_{\ve}\cdot\de_t(v-u_{\ve})\right) \, dx \, dt}_S
\end{align}
Since,
\begin{align*}
P =\int_0^{\theta} \iint\limits_{\RR^N\times\RR^N}& (1-\xi_{\theta})\frac{H(x,y,u_{\ve}(x,t)-u_{\ve}(y,t))}{|x-y|^{N}} \,dx \, dy\, dt \\
&\qquad+ \int_{T-\theta}^{T} \iint\limits_{\RR^N\times\RR^N} (1-\xi_{\theta})\frac{H(x,y,u_{\ve}(x,t)-u_{\ve}(y,t))}{|x-y|^{N}} \,dx \, dy\, dt
\end{align*}
using Lemma $\ref{lem:spacebound}$, for $\delta \geq \ve$ we get
\[
0 \leq P \leq 4\theta e \Lambda. 
\]
Next, 
\[
Q = \int_0^{T}\int_{\Omega} \xi_{\theta}\de_t v\cdot(v-u_{\ve}) \,dx\,dt\ - \frac{1}{2}\int_0^{T}\int_{\Omega} \xi_{\theta}\de_t |v-u_{\ve}|^2 \,dx\,dt.
\]
Integrating by parts,
\[
- \frac{1}{2}\int_0^{T}\int_{\Omega} \xi_{\theta}\de_t |v-u_{\ve}|^2 \,dx\,dt\ = \frac{1}{2\theta}\int_0^{T}\int_{\Omega}|v-u_{\ve}|^2\,dx\,dt\ - \frac{1}{2\theta}\int_{T-\theta}^{T}\int_{\Omega}|v-u_{\ve}|^2\,dx\,dt\
\]
We also have, recalling Remark $\ref{rem:yadayada}$

\begin{align*}
    \frac{1}{2\theta}\int_0^{T}\int_{\Omega}|v-u_{\ve}|^2\,dx\,dt &\leq \left[\left(\frac{1}{2\theta}\int_0^{\theta}\int_{\Omega}|v-u_{0}|^2\,dx\,dt\right)^{1/2}+\left(\frac{1}{2\theta}\int_0^{\theta}\int_{\Omega}|u_{\ve}-u_0|^2\,dx\,dt\right)^{1/2}\right]^{2}\\
    &\leq
    \left[\left(\frac{1}{2\theta}\int_0^{\theta}\int_{\Omega}|v-u_{0}|^2\,dx\,dt\right)^{1/2}+\frac{\sqrt{\Lambda}}{2}\sqrt{\theta}\right]^{2}
\end{align*}

Passing to the weak limit as $\ve \rightarrow 0+$ in~\cref{PQRS} and using the monotonicity of $H$ followed by convexity of $H$ in $\xi$ we get
\begin{align*}
    \int_0^{T} \iint\limits_{\RR^N\times\RR^N} \frac{H(x,y,u(x,t)-u(y,t))}{|x-y|^{N}} \,dx \, dy\, dt \leq \liminf_{\ve \rightarrow 0+} \int_0^{T} \iint\limits_{\RR^N\times\RR^N} \frac{H(x,y,u_{\ve}(x,t)-u_{\ve}(y,t))}{|x-y|^{N}} \,dx \, dy\, dt \\
    \leq 
    \left[\left(\frac{1}{2\theta}\int_0^{\theta}\int_{\Omega}|v-u_{0}|^2\,dx\,dt\right)^{1/2}+\frac{\sqrt{\Lambda}}{2}\sqrt{\theta}\right]^{2} - \frac{1}{2\theta}\int_{T-\theta}^{T}\int_{\Omega}|v-u|^2\,dx\,dt\\
    + \int_0^{T}\int_{\Omega} \xi_{\theta}\de_t v\cdot(v-u) \,dx\,dt 
    + \int_0^{T} \iint\limits_{\RR^N\times\RR^N}\xi_{\theta}\frac{H(x,y,v(x,t)-v(y,t))}{|x-y|^{N}}\,dx \, dy\, dt \\ + 4\theta e\Lambda
\end{align*}
Finally, letting $\theta \rightarrow 0+$ we get
\[
\begin{split}
    \int_0^T \iint\limits_{\RR^N\times\RR^N}\frac{H(x,y,u(x,t)-u(y,t))}{|x-y|^N} \,dx \,dy \,dt \leq \int_0^T\frac{1}{2}\norm{v(\cdot,0) - u_0(\cdot)}_{L^2(\Omega)}^2 - \frac{1}{2}\norm{v(\cdot,T) - u(\cdot,T)}_{L^2(\Omega)}^2 \\
    + \int_0^T \int_{\Omega}\de_tv\cdot(v-u) + 
    \int_0^T\iint\limits_{\RR^N\times\RR^N}\frac{H(x,y,v(x,t)-v(y,t))}{{|x-y|^N}} \,dx \,dy \,dt
\end{split}
\]
which is $\eqref{eq:vareq}$. This completes the proof of  Theorem~\ref{mainthm:e}


\section{Uniqueness of variational solutions} In this section, we will show that a variational solution is unique provided that the function $\xi\mapsto H(x,y,\xi)$ is strictly convex proving Theorem~\ref{mainthm:u}. 

To this end, we suppose that $u_1,u_2\in L^p(0,T;W^{s,p}(\RR^N))\cap C^0([0,T];L^2(\Omega))$ for any $T>0$ are two different variational solutions. We add the two variational inequalities~\cref{defvar} corresponding to $u_1$ and $u_2$, to obtain

\begin{align}\label{addtwomin}
    \sum_{i=1}^2\int_0^T\iint\limits_{\mathbb{R}^N\times\mathbb{R}^N} &\frac{H(x,y,u_i(x,t)-u_i(y,t))}{|x-y|^N}\,dx\,dy\,dt
    \leq 2\int_0^T\int_{\Omega} \de_t v\,(v-w)\,dx\,dt\nonumber\\
    & + 2\int_0^T\iint\limits_{\mathbb{R}^N\times\mathbb{R}^N} \frac{H(x,y,v(x,t)-v(y,t))}{|x-y|^N}\,dx\,dy\,dt+\norm{v(\cdot,0)-u_0}_{L^2(\Om)}^2,
\end{align}
where we have dropped the negative term from the right hand side and we have defined $w=\frac{u_1+u_2}{2}$. We would like to use $w$ as a comparison function $v$, however the definition of variational solutions does not admit the requisite regularity of $\de_t w\in L^2(0,T;\Omega)$.

To overcome this, we make use of $v=[w]_h$ with $v_0=u_0$. Then, by Lemma~\ref{propertiesofmolli}, we get \[[w]_h\in L^p(0,T;W^{s,p}(\RR^N)) \mbox{ with } \de_t[w]_h\in L^2(0,T;\Omega)\] and $[w]_h=u_0$ on $(\Omega^c\times (0,\infty))\cup (\Omega\times\{0\})$.

Now, choosing $v=[w]_h$ in~\cref{addtwomin}, we get

\begin{align}\label{addtwomin2}
    \sum_{i=1}^2\int_0^T\iint\limits_{\mathbb{R}^N\times\mathbb{R}^N} &\frac{H(x,y,u_i(x,t)-u_i(y,t))}{|x-y|^N}\,dx\,dy\,dt
    \leq 2\int_0^T\int_{\Omega} \de_t v\,(v-w)\,dx\,dt\nonumber\\
    & + 2\int_0^T\iint\limits_{\mathbb{R}^N\times\mathbb{R}^N} \frac{H(x,y,[w]_h(x,t)-[w]_h(y,t))}{|x-y|^N}\,dx\,dy\,dt=2I_h+2{II}_h.
\end{align}

By Lemma~\ref{propertiesofmolli}, we have
\begin{align}\label{addminlim1}
    I_h=-\frac{1}{h}\int_0^T\int_\Omega |[w]_h-w|^2\,dz\leq 0.
\end{align}

For ${II}_h$, by convexity, we have
\begin{align*}
    II_h\leq \frac{1}{2}\int_0^T \iint\limits_{\mathbb{R}^N\times\mathbb{R}^N} \frac{H(x,y,[u_1]_h(x,t)-[u_1]_h(y,t))+H(x,y,[u_2]_h(x,t)-[u_2]_h(y,t))}{|x-y|^N}\,dx\,dy\,dt<\infty.
\end{align*}
Therefore, $\frac{H(x,y,[w]_h(x,t)-[w]_h(y,t))}{|x-y|^N}\in L^1(0,T;\RR^N\times\RR^N)$. As a result, by Theorem~\ref{convergenceoffunctionals}, we get 
\begin{align}\label{addminlim2}
    \lim_{h\to 0} II_h = \int_0^T \iint\limits_{\mathbb{R}^N\times\mathbb{R}^N} \frac{H\left(x,y,\frac{(u_1+u_2)(x,t)}{2}-\frac{(u_1+u_2)(y,t)}{2}\right)}{|x-y|^N}\,dx\,dy\,dt.
\end{align}

Combining~\cref{addtwomin2},~\cref{addtwomin2} and~\cref{addminlim2}, we get by passing to the limit as $h\to 0$,

\begin{align}\label{addtwomin3}
    \sum_{i=1}^2\int_0^T\iint\limits_{\mathbb{R}^N\times\mathbb{R}^N} &\frac{H(x,y,u_i(x,t)-u_i(y,t))}{|x-y|^N}\,dx\,dy\,dt\\
    &\leq  2\int_0^T \iint\limits_{\mathbb{R}^N\times\mathbb{R}^N} \frac{H\left(x,y,\frac{(u_1+u_2)(x,t)}{2}-\frac{(u_1+u_2)(y,t)}{2}\right)}{|x-y|^N}\,dx\,dy\,dt\nonumber\\
    &< \sum_{i=1}^2\int_0^T\iint\limits_{\mathbb{R}^N\times\mathbb{R}^N} \frac{H(x,y,u_i(x,t)-u_i(y,t))}{|x-y|^N}\,dx\,dy\,dt
\end{align}

In the last step, we have used strict convexity of $H$ in $\xi$ which leads to the desired contradiction. This finishes the proof of the uniqueness of variational solutions to~\cref{maineq}.


\section{Variational Solutions are Parabolic Minimizers}
A notion of parabolic Q-minimizers was introduced by Weiser in~\cite{wieserParabolicQminimaMinimal1987}, which has been useful to study regularity such as higher integrability~\cite{parviainenGlobalHigherIntegrability2008} for parabolic equations, even in measure metric settings~\cite{kinnunenParabolicComparisonPrinciple2015}. We define here a notion of parabolic minimizer for nonlocal parabolic equations of the form~\cref{maineq} and prove that variational solutions as in~\cref{defvar} are parabolic minimizers, although the converse might require additional time regularity.

\begin{definition}\label{defparmin}
A measurable function $u:\RR^N\times(0,\infty)\to\RR$ is called a parabolic minimizer for the equation \cref{maineq} if $u\in L^p(0,T;W^{s,p}(\RR^N))$ for all $T>0$ and $u-u_0\in L^p(0,T;W^{s,p}_0(\Omega))$, and the following minimality condition holds:
\begin{align*}
    \int_0^T \int_{\Om} u \de_t\phi\,dx\,dt &+ \int_0^T\iint\limits_{\mathbb{R}^N\times\mathbb{R}^N} \frac{H(x,y,u(x,t)-u(y,t))}{|x-y|^N}\,dx\,dy\,dt\\
    &\qquad\leq \int_0^T\iint\limits_{\mathbb{R}^N\times\mathbb{R}^N} \frac{H(x,y,(u+\phi)(x,t)-(u+\phi)(y,t))}{|x-y|^N}\,dx\,dy\,dt,
\end{align*} whenever $T>0$ and $\phi\in C_0^\infty(\Omega\times(0,T))$
\end{definition}

\begin{theorem}\label{varisparmin}
    If $u$ is a variational solution of~\cref{maineq} in the sense of~\ref{defvar}, it is also a parabolic minimizer in the sense of Definition~\ref{defparmin}.
\end{theorem}

\begin{proof}
    Let $T>0$ be fixed and $\phi\in C_0^\infty(\Om_T)$. We would like to choose $v=u+s\phi$ as a comparison function in~\cref{defvar} but this may not be admissible. We overcome this by mollification in time. Choose $v=v_h=[u]_h+s[\phi]_h$ with $u_0$ and $\phi_0=0$ respectively. 

    Now, observe that
    \begin{align*}
        \int_0^T \int_{\Om} \de_t v&\, (v_h-u)\,dx\,dt  \\
       & = \int_0^T \int_{\Om} \de_t [u]_h\, ([u]_h-u)\,dx\,dt + s\int_0^T \int_{\Om} \de_t [u]_h\, [\phi]_h\,dx\,dt + s\int_0^T \int_{\Om} \de_t [\phi]_h\, ([v]_h-u)\,dx\,dt\\
        &= -\frac{1}{h}\int_0^T \int_\Omega |[u]_h-h|^2 - s\int_0^T \int_{\Om} [u]_h\, \de_t [\phi]_h\,dx\,dt \\
        &\qquad +s\int_0^T \int_{\Om} \de_t [\phi]_h\, ([v]_h-u)\,dx\,dt + s\int_{\Om}  [u]_h(T)\, [\phi]_h(T)\,dx\\
        &\leq s\int_0^T \int_{\Om} \de_t [\phi]_h\, (s[\phi]_h-u)\,dx\,dt + s\int_{\Om}  [u]_h(T)\, [\phi]_h(T)\,dx
    \end{align*}
In the second equality, we have used Lemma~\ref{propertiesofmolli} and integration by parts. Plugging this in the minimality condition, we get 
\begin{align*}
    \int_0^T &\iint\limits_{\mathbb{R}^N\times\mathbb{R}^N}\frac{H(x,y,u(x,t)-u(y,t))}{|x-y|^N}\,dx\,dy\,dt\\
    &\leq s\int_0^T \int_{\Om} \de_t [\phi]_h\, (s[\phi]_h-u)\,dx\,dt + s\int_{\Om}  [u]_h(T)\, [\phi]_h(T)\,dx \\
    &\qquad+ \int_0^T \iint\limits_{\mathbb{R}^N\times\mathbb{R}^N}\frac{H(x,y,v_h(x,t)-v_h(y,t))}{|x-y|^N}\,dx\,dy\,dt - \frac{1}{2}\norm{(v_h-u)(\cdot,T)}^2_{L^2(\Om)}.
\end{align*} 

Dropping the non-positive term from the right hand side and passing to the limit as $h\to 0$ by Lemma~\ref{propertiesofmolli} and Theorem~\ref{convergenceoffunctionals}, we get 
\begin{align*}
    \int_0^T &\iint\limits_{\mathbb{R}^N\times\mathbb{R}^N}\frac{H(x,y,u(x,t)-u(y,t))}{|x-y|^N}\,dx\,dy\,dt\\
    &\leq s\int_0^T \int_{\Om} \de_t \phi\, (s\phi-u)\,dx\,dt + \int_0^T \iint\limits_{\mathbb{R}^N\times\mathbb{R}^N}\frac{H(x,y,(u+s\phi)(x,t)-(u+s\phi)(y,t))}{|x-y|^N}\,dx\,dy\,dt\\
    & \leq s\int_0^T \int_{\Om} \de_t \phi\, (s\phi-u)\,dx\,dt + (1-s)\int_0^T \iint\limits_{\mathbb{R}^N\times\mathbb{R}^N}\frac{H(x,y,u(x,t)-u(y,t))}{|x-y|^N}\,dx\,dy\,dt \\
    &\qquad + s\int_0^T \iint\limits_{\mathbb{R}^N\times\mathbb{R}^N}\frac{H(x,y,((u+\phi)(x,t)-(u+\phi)(y,t))}{|x-y|^N}\,dx\,dy\,dt.
\end{align*}

In the second inequality, we have used convexity of $H(x,y,\xi)$ in $\xi$. Now, we subtract the second term on the right hand side from both sides, followed by dividing by $s$, to obtain

\begin{align*}
    \int_0^T \iint\limits_{\mathbb{R}^N\times\mathbb{R}^N}&\frac{H(x,y,u(x,t)-u(y,t))}{|x-y|^N}\,dx\,dy\,dt\\
    & \leq \int_0^T \int_{\Om} \de_t \phi\, (s\phi-u)\,dx\,dt + \int_0^T \iint\limits_{\mathbb{R}^N\times\mathbb{R}^N}\frac{H(x,y,((u+\phi)(x,t)-(u+\phi)(y,t))}{|x-y|^N}\,dx\,dy\,dt.
\end{align*}

Finally, we let $s\to 0$ to get the desired conclusion.

\end{proof}

\section{Comparison with weak solutions}

In this section, we compare the notion of weak solutions to that of variational solutions. 

We will first show that if the function $\xi\mapsto H(x,y,\xi)$ is $C^1$ and satisfies a comparable growth condition from above namely

\begin{align}
    \label{eq:bound_above_H} H(x,y,\xi) \leq C\left(\frac{|\xi|}{|x-y|^s}\right)^p
\end{align}
then a variational solution is a weak solution.

\begin{definition}
    Let $\Om$ be an open bounded subset of $\RR^N$. Suppose that $H$ satisfies the assumptions~\cref{eq:bound_below_H} ~\cref{eq: cvx_H} and ~\cref{eq:bound_above_H} and let the time-independent Cauchy-Dirichlet data $u_0$ satisfy~\cref{datahypo}.

     We say that \[u\in L^p(0,T;W^{s,p}(\mathbb{R}^N))\cap C^0(0,T;L^2(\Om)), \mbox{ such that } u-g\in L^p(0,T;W^{s,p}_0(\Om))\] and $\de_t u\in L^{p'}(0,T;(W^{s,p}_0(\Om))^*)$, is a weak solution to~\cref{maineq} if for every $\tau\in(0,T]$, we have
     \begin{align}
         \label{defweak}
         \int_{0}^{\tau}\iint\limits_{\mathbb{R}^N\times\mathbb{R}^N} \frac{\de_{\xi}H(x,y,u(x,t)-u(y,t))(\phi(x,t)-\phi(y,t))}{|x-y|^N}\,dx\,dy\,dt+\int_{0}^{\tau} \iprod{\de_t u}{\phi}_{W^{s,p}(\RR^N)}\,dt =  0,
     \end{align} for all $\phi\in C_0^\infty(\Om_\tau)$.
\end{definition}




\subsection{Time derivative of variational solutions}
In this subsection, we will prove that a variational solution $u$ has a time derivative \[\de_t u\in L^{p'}(0,T;(W^{s,p}_0(\Om))^*).\] We will use the fact that variational solutions are parabolic minimizers as proved in Theorem~\ref{varisparmin}. Now, for $\phi\in C_0^\infty(\Om_T)$, taking $s\phi$ instead of $\phi$ in~\ref{defparmin}, we obtain

\begin{align}
    \bigg| \int_0^T \int_{\Om} u\cdot\de_t\phi &\,dx\,dt\bigg| \leq \nonumber\\
    &\left| \int_0^T\iint\limits_{\mathbb{R}^N\times\mathbb{R}^N} \frac{1}{s}\frac{H(x,y,(u+s\phi)(x,t)-(u+s\phi)(y,t))-H(x,y,u(x,t)-u(y,t))}{|x-y|^N} \right|
\end{align}
By passing to the limit as $s\to 0$, we get 
\begin{align*}
    \left| \int_0^T \int_{\Om} u\cdot\de_t\phi \,dx\,dt\right| & \leq \int_0^T\iint\limits_{\mathbb{R}^N\times\mathbb{R}^N} \frac{\left|\de_{\xi}H(x,y,u(x,t)-u(y,t))\right|\left|(\phi(x,t)-\phi(y,t))\right|}{|x-y|^N}\,dx\,dy\,dt\\
    & \leq C\int_0^T\iint\limits_{\mathbb{R}^N\times\mathbb{R}^N} \frac{|u(x,t)-u(y,t)|^{p-1}\left|(\phi(x,t)-\phi(y,t))\right|}{|x-y|^{N+sp}}\,dx\,dy\,dt\\
    & \leq C \norm{u}_{L^p(0,T;W^{s,p}(\RR^N))}^{p-1}\norm{\phi}_{L^p(0,T;W^{s,p}(\RR^N))}\\
    & \leq C \norm{u}_{L^p(0,T;W^{s,p}(\RR^N))}^{p-1}\norm{\phi}_{L^p(0,T;W^{s,p}_0(\Om))}.
\end{align*}

In the second inequality, we have used the fact that if $H$ satisfies~\ref{eq:bound_below_H} and \ref{eq:bound_above_H}, then $\de_{\xi}H$ satisfies \begin{align*}\de_\xi H(x,y,\xi)\leq C\frac{|\xi|^{p-1}}{|x-y|^{sp}}.\end{align*} For a proof, see~\cite[Lemma~2.2]{marcelliniRegularityMinimizersIntegrals1989}. In the third inequality, we have used H\"older's inequality. Now, since $C_0^{\infty}(\Om_T)$ is dense in $L^p(0,T;W^{s,p}_0(\Om))$, we get that $\de_t u \in L^{p'}(0,T;(W^{s,p}_0(\Om))^*)$.

\subsection{Passage to weak solutions} In this subsection, we prove that under~\cref{eq:bound_above_H}, variational solutions are weak solutions as well. Once again, we use the definition of parabolic minimizers as it is more convenient. In particular, taking $s\phi$ instead of $\phi$ in Definition~\ref{defparmin}, and taking limit as $s\to 0$, we get 

\begin{align*}
    -\int_0^T \int_\Om \de_t u\, \phi\,dx\,dt = \int_0^T \int_\Om & u\,\de_t\phi\,dx\,dt\\
    &\qquad \leq \int_0^T \iint\limits_{\mathbb{R}^N\times\mathbb{R}^N} \frac{\de_{\xi}H(x,y,u(x,t)-u(y,t))(\phi(x,t)-\phi(y,t))}{|x-y|^N}\,dx\,dy\,dt.
\end{align*}

If we take $-s\phi$ instead of $\phi$, we get the reverse inequality. This completes the proof.

\subsection{Passage from weak to variational} 

We will now show that for the functional,
\begin{equation}\label{eq:double_phase_H}
    H(x,y,\xi) = \frac{1}{p}\left(\frac{|\xi|}{|x-y|^s}\right)^p
\end{equation}
any weak solution with $\de_t u\in L^{p'}(0,T;(W^{s,p}_0(\Om))^*)$ must necessarily be a variational solution in one of two cases, either we have $p>\frac{2N}{2s+N}$ or we have $\de_t u\in L^2(\Om_T)$. 

The condition on time derivative $\de_t u\in L^{p'}(0,T;(W^{s,p}_0(\Om))^*)$ is always satisfied as per the standard theory of monotone operators in Banach space ~\cite{stromqvistLocalBoundednessSolutions2019}.  To proceed with the argument, let $\phi = v-u$ be our test function where $v$ is a comparison function. The function $\phi$ is a valid test function since $C_c^\infty(\Om_T)$ is dense in $L^p(0,T;W^{s,p}_0(\Om))$. We note that
    \begin{align}\label{onest0}
    |(u+\phi)(x) - (u+\phi)(y)|^p - |u(x)-u(y)|^p \geq p |u(x) - u(y)|^{p-2}(u(x) - u(y))(\phi(x)-\phi(y))
    \end{align}
and compute 
\begin{align}\label{oneest1}
  \int_0^T \inp{\de_t u}{\phi} &=
  \int_0^T \inp{\de_t v}{v-u} - \int_0^T \inp{\de_t (v-u)}{v-u}  \nonumber\\
  &= \int_0^T \inp{\de_t v}{v-u} - \frac{1}{2}\int_0^T \frac{d}{dt}\norm{v-u}_{L^2}^2 \nonumber\\
  &= \int_0^T \inp{\de_t v}{v-u} - \norm{(v-u)(\cdot,T)}_{L^2}^2 + \norm{(v-u)(\cdot,0)}_{L^2}^2 .
\end{align}
The second equality holds in one of two cases. Either $p>\frac{2N}{2s+N}$ which ensures the existence of the Gelfand triple $W_0^{s,p}(\Om)\xhookrightarrow{}L^2(\Om)\xhookrightarrow{}W^{-s,p'}(\Om)$, so that \cref{integratebypartsintime} can be applied. Or we assume that the weak solution $u$ satisfies $\de_t u\in L^2(\Om_T)$.

Now, by the definition of weak solution \cref{defweak}, \cref{oneest1} and applying \cref{onest0}, we obtain that
\begin{align*}
\int_0^T \inp{\de_t v}{v-u} - \norm{(v-u)(\cdot,T)}_{L^2}^2 + \norm{(v-u)(\cdot,0)}_{L^2}^2 \geq \int_{0}^{T}\iint\limits_{\mathbb{R}^N\times\mathbb{R}^N}\frac{H(x,y,u(x,t)-u(y,t))}{|x-y|^{N}} \,dx \, dy\, dt \\
-\int_{0}^{T}\iint\limits_{\mathbb{R}^N\times\mathbb{R}^N}\frac{H(x,y,v(x,t)-v(y,t))}{|x-y|^{N}} \,dx \, dy\, dt.  
\end{align*}

Rearranging the above inequality shows that $u$ is a variational solution.
\appendix
\section{Mollification in Time}\label{Molli}
\renewcommand{\thesection}{\Alph{section}}
In the definition of variational solutions, the test functions or comparison functions have additional time regularity compared to the solutions, therefore the variational solutions themselves cannot be used as comparison functions. To fix this, we need a smoothening in time. Let $\Omega$ be an open subset of $\mathbb{R}^N$. For $T>0$, $v\in L^1(\Omega_T)$, $v_0\in L^1(\Omega)$ and $h\in (0,T]$, we define
\begin{align}
    \label{timemolli}
    [v]_h(\cdot,t)=e^{-\frac{t}{h}}v_0 + \frac{1}{h}\int_0^t e^{\frac{s-t}{h}}v(\cdot,s)\,ds,
\end{align} for $t\in [0,T]$. The basic properties of time mollification were proved earlier in~\cite{kinnunenPointwiseBehaviourSemicontinuous2006,bogeleinParabolicSystemsQGrowth2013}. We state them below for easy reference and prove the ones that are new.

\begin{proposition}(\cite[Lemma~B.1]{bogeleinParabolicSystemsQGrowth2013})\label{generalmolliX}
Let $X$ be a Banach space and assume that $v_0\in X$, and moreover $v\in L^r(0,T;X)$ for some $1\leq r\leq\infty$. Then, the mollification in time defined by~\eqref{timemolli} belongs to $L^r(0,T;X)$ and
\begin{align}
    ||[v]_h||_{L^r(0,T;X)}\leq ||v||_{L^r(0,t_0;X)}+ \left(\frac{h}{r}\left(1-e^{-\frac{t_0 r}{h}}\right)\right)^{\frac{1}{r}}||v_0||_{X},
\end{align} for any $t_0\in(0,T)$. Moreover, we have
\begin{align}
    \de_t[v]_h\in L^r(0,T;X)\mbox{ and } \de_t[v]_h=-\frac{1}{h}([v]_h-v).
\end{align}
\end{proposition}

\begin{lemma}(\cite[Lemma~2.2]{bogeleinExistenceEvolutionaryVariational2014})
\label{propertiesofmolli}
Let $\Omega$ be an open subset of $\mathbb{R}^N$. Suppose that $v\in L^1(\Omega_T)$ and $v_0\in L^1(\Omega)$. Then, the mollification in time as defined in~\eqref{timemolli} satisfies the following properties:
\begin{itemize}
    \item[(i)] Assume that $v\in L^p(\Omega_T)$ and $v_0\in L^p(\Omega)$ for some $p\geq 1$. Then, it holds true that $[v]_h\in L^p(\Omega_T)$ and the following quantitative bound holds.
    \begin{align}
    ||[v]_h||_{L^p(\Omega_T)}\leq ||v||_{L^p(\Omega_T)}+ h^{1/p}||v_0||_{L^p(\Omega)}.
\end{align}
Moreover, $[v]_h\to v$ in $L^p(\Omega_T)$ as $h\to 0$.
\item[(ii)] Assume that $v\in L^p(0,T;W^{s,p}(\Omega))$ and $v_0\in W^{s,p}(\Omega)$ for some $p > 1$ and $s\in (0,1]$. Then, it holds true that $[v]_h\in L^p(0,T;W^{s,p}(\Omega))$ and the following quantitative bound holds.
    \begin{align}
    ||[v]_h||_{L^p(0,T;W^{s,p}(\Omega))}\leq ||v||_{L^p(0,T;W^{s,p}(\Omega))}+ h^{1/p}||v_0||_{W^{s,p}(\Omega)}.
\end{align}
Moreover, $[v]_h\to v$ in $L^p(0,T;W^{s,p}(\Omega))$ as $h\to 0$.
\item[(iii)] Suppose that $v\in L^p(0,T;W^{s,p}_0(\Omega))$ and $v_0\in W^{s,p}_0(\Omega)$ for some $p > 1$ and $s\in (0,1]$. Then, it holds true that $[v]_h\in L^p(0,T;W^{s,p}_0(\Omega))$.
\item[(iv)] Suppose that $v\in C^0(0,T;L^2(\Omega))$ and $v_0\in L^2(\Omega)$. Then, it holds true that $[v]_h\in C^0(0,T;L^2(\Omega))$, $[v]_h(\cdot,0)=v_0$. Moreover, $[v]_h\to v$ in $C^0(0,T;L^2(\Omega))$ as $h\to 0$.
\item[(v)] Suppose that $v\in L^\infty(0,T;L^2(\Omega))$ and $v_0\in L^2(\Omega)$. Then, it holds true that $\de_t [v]_h\in L^\infty(0,T;L^2(\Omega))$. Moreover, $\de_t [v]_h=-\frac{1}{h}([v]_h-v)$.
\item[(vi)] Suppose that $\de_t v\in L^2(\Omega_T)$ then $\de_t [v]_h\to \de_t v$ in $L^2(\Omega_T)$ as $h\to 0$. Moreover, the inequality $||\de_t [v]_h||_{L^2(\Omega_T)}\leq ||\de_t v||_{L^2(\Omega_T)}$ holds true.
\end{itemize}
\end{lemma}

\begin{proof}
The proofs of statements $(i), (iv), (v)$ and $(vi)$ are the same as in~\cite[Lemma~B.2]{bogeleinParabolicSystemsQGrowth2013}. We indicate the modification required for $(ii)$ and $(iii)$.
\paragraph{Proof of $(ii)$} The case $s=1$ is covered in~\cite[Lemma~B.2]{bogeleinParabolicSystemsQGrowth2013}. Therefore, we assume that $s\in (0,1)$. To begin with, $[v]_h\in L^p(0,T;W^{s,p}(\Omega))$ follows from Proposition~\ref{generalmolliX}. To prove the convergence, observe that the inclusion $[v]_h\in L^p(0,T;W^{s,p}(\Omega))$ implies that  the function defined by $f_h(x,y,t)=\frac{|[v]_h(x,t)-[v]_h(y,t)|}{|x-y|^{\frac{N}{p}+s}}$ satisfies $f_h\in L^p(0,T;L^p(\Omega\times\Omega))$. 

The convergence of $[v]_h$ to $v$ in $L^p(0,T;\RR^N)$ as $h\to 0$ guarantees pointwise a.e. convergence for a subsequence. Since this is true for any subsequence, we get the convergence $f_h(x,y,t)^p\to \dfrac{|v(x,t)-v(y,t)|^p}{|x-y|^{N+ps}}$ as $h\to 0$ pointwise a.e. $x,t\in\RR^N\times (0,T)$ due to continuity of $f_h$ in $v$.

Now, we will interpret the mollification as a mean with respect to the measure $e^{\frac{s-t}{h}}\,ds$ similar to \cite[Lemma~2.3]{bogeleinExistenceEvolutionaryVariational2014}. This will allow us to use the convexity of $f_h^p$ as the $p^{th}$ power of the ratio $\frac{|[v]_h(x,t)-[v]_h(y,t)|}{|x-y|^{\frac{N}{p}+s}}$. 

Observe the following chain of inequalities:
\begin{align*}
    f_h^p(x,y,t)&=\frac{|[v]_h(x,t)-[v]_h(y,t)|^p}{|x-y|^{N+sp}}=\frac{\left|e^{-t/h}(v_0(x)-v_0(y))+\frac{1-e^{-t/h}}{h(1-e^{-t/h})}\int_{0}^{t}(v(x,t)-v(y,t))e^{\frac{s-t}{h}\,ds}\right|^p}{|x-y|^{N+ps}}\\
    &\leq e^{-t/h}\frac{\left|v_0(x)-v_0(y)\right|^p}{|x-y|^{N+ps}}+\left(1-e^{-t/h}\right)\frac{\left|\frac{1}{h(1-e^{-t/h})}\int_{0}^{t}(v(x,t)-v(y,t))e^{\frac{s-t}{h}\,ds}\right|^p}{|x-y|^{N+ps}}\\
    &\leq e^{-t/h}\frac{\left|v_0(x)-v_0(y)\right|^p}{|x-y|^{N+ps}}+\frac{1}{h}\frac{\int_{0}^{t}\left|(v(x,t)-v(y,t))e^{\frac{s-t}{h}}\right|^p\,ds}{|x-y|^{N+ps}}
    =\left[\frac{|v(x,t)-v(y,t)|^p}{|x-y|^{N+ps}}\right]_h,
\end{align*} where we have used the Jensen's inequality in the last inequality.

Since $\frac{|v(x,t)-v(y,t)|^p}{|x-y|^{N+ps}}\in L^1(0,T;L^1(\Omega\times\Omega))$ and $\frac{|v_0(x)-v_0(y)|^p}{|x-y|^{N+ps}}\in L^1(\Omega\times\Omega)$, we conclude that \[\left[\frac{|v(x,t)-v(y,t)|^p}{|x-y|^{N+ps}}\right]_h\in L^1(0,T;L^1(\Omega\times\Omega))\] by $(i)$ with the additional bound
\begin{align*}
    \left|\left|\left[\frac{|v(x,t)-v(y,t)|^p}{|x-y|^{N+ps}}\right]_h\right|\right|_{L^1(0,T;L^1(\Omega\times\Omega))}\leq& \left|\left|\frac{|v(x,t)-v(y,t)|^p}{|x-y|^{N+ps}}\right|\right|_{L^1(0,T;L^1(\Omega\times\Omega))}\\
    &\qquad+h\left|\left|\frac{|v_0(x)-v_0(y)|^p}{|x-y|^{N+ps}}\right|\right|_{L^1(\Omega\times\Omega)}.
\end{align*}

Since $h\left|\left|\frac{|v_0(x)-v_0(y)|^p}{|x-y|^{N+ps}}\right|\right|_{L^1(\Omega\times\Omega)}\to 0$ as $h\to 0$, by a version of dominated convergence theorem, we conclude that
\begin{align*}
    \lim_{h\to 0}\int_0^T\iint\limits_{\Om\times\Om} f^p_h(x,y,t)\,dx\,dy\,dt= \int_0^T\iint\limits_{\Om\times\Om}\frac{|v(x,t)-v(y,t)|^p}{|x-y|^{N+ps}}\,dx\,dy\,dt.
\end{align*} This finishes the proof of $(ii)$.
\paragraph{Proof of $(iii)$} The statement may be proved by density of $C_0^\infty(\Omega)$ functions in $W^{s,p}_0(\Omega)$. In particular, if $\phi_{\varepsilon}\in L^p(0,T;C_0^\infty(\Omega))$ is an approximating sequence for a function $v\in L^p(0,T;W^{s,p}(\Omega))$, then $[\phi_{\varepsilon}]_h$ is an approximating sequence for the function $[v]_h\in L^p(0,T;W^{s,p}(\Omega))$.
\end{proof}

We finish this appendix by proving the following theorem.

\begin{theorem}
\label{convergenceoffunctionals}
Let $T>0$, and assume that $v\in L^1(0,T;W^{s,1}(\mathbb{R}^N))$ with
\begin{align*}
    \frac{H(x,y,v(x,t)-v(y,t))}{|x-y|^N}\in L^1(0,T;L^1(\mathbb{R}^N\times\mathbb{R}^N)),
\end{align*} and $v_0\in W^{s,1}(\mathbb{R}^N)$, with
\begin{align*}
    \frac{H\left(x,y,v_0(x)-v_0(y)\right)}{|x-y|^N}\in L^1(\mathbb{R}^N\times\mathbb{R}^N).
\end{align*} Then, we have
\begin{align*}
    \frac{H\left(x,y,[v]_h(x,t)-[v]_h(y,t)\right)}{|x-y|^N}\in L^1(0,T;L^1(\mathbb{R}^N\times\mathbb{R}^N)).
\end{align*}
Moreover,
\begin{align*}
    \lim_{h\to 0}\int_0^T\iint\limits_{\mathbb{R}^N\times\mathbb{R}^N}\frac{H\left(x,y,[v]_h(x,t)-[v]_h(y,t)\right)}{|x-y|^N}\,dx\,dy\,dt=\int_0^T\iint\limits_{\mathbb{R}^N\times\mathbb{R}^N}\frac{H\left(x,y,v(x,t)-v(y,t)\right)}{|x-y|^N}\,dx\,dy\,dt
\end{align*}
\end{theorem}

\begin{proof}
The proof is the same as that of \cite[Lemma~2.3]{bogeleinExistenceEvolutionaryVariational2014}. The only difference is we use continuity and convexity of $\xi\to H(x,y,\xi)$ and use convergence in $L^1(0,T;L^1(\mathbb{R}^N\times\mathbb{R}^N))$. Indeed, the proof is similar to the one given for Lemma~\ref{propertiesofmolli} (ii).
\end{proof}

\section*{Acknowledgements} 
The authors would like to thank Karthik Adimurthi for introducing us to this subject and for illuminating discussions. The authors were supported by the Department of Atomic Energy,  Government of India, under	project no.  12-R\&D-TFR-5.01-0520. 

\bibliography{MyLibrary}

\begin{thebibliography}{46}
\providecommand{\natexlab}[1]{#1}
\providecommand{\url}[1]{\texttt{#1}}
\expandafter\ifx\csname urlstyle\endcsname\relax
  \providecommand{\doi}[1]{doi: #1}\else
  \providecommand{\doi}{doi: \begingroup \urlstyle{rm}\Url}\fi

\bibitem[B{\"o}gelein et~al.(2013)B{\"o}gelein, Duzaar, and
  Marcellini]{bogeleinParabolicSystemsQGrowth2013}
Verena B{\"o}gelein, Frank Duzaar, and Paolo Marcellini.
\newblock Parabolic {{Systems}} with p, q-{{Growth}}: {{A Variational
  Approach}}.
\newblock \emph{Archive for Rational Mechanics and Analysis}, 210\penalty0
  (1):\penalty0 219--267, October 2013.
\newblock ISSN 1432-0673.
\newblock \doi{10.1007/s00205-013-0646-4}.
\newblock URL \url{10.1007/s00205-013-0646-4}.

\bibitem[B{\"o}gelein et~al.(2014)B{\"o}gelein, Duzaar, and
  Marcellini]{bogeleinExistenceEvolutionaryVariational2014}
Verena B{\"o}gelein, Frank Duzaar, and Paolo Marcellini.
\newblock Existence of evolutionary variational solutions via the calculus of
  variations.
\newblock \emph{Journal of Differential Equations}, 256\penalty0 (12):\penalty0
  3912--3942, June 2014.
\newblock ISSN 0022-0396.
\newblock \doi{10.1016/j.jde.2014.03.005}.
\newblock URL
  \url{https://www.sciencedirect.com/science/article/pii/S0022039614001144}.

\bibitem[B{\"o}gelein et~al.(2019)B{\"o}gelein, Duzaar, Sch{\"a}tzler, and
  Scheven]{bogeleinExistenceEvolutionaryProblems2019}
Verena B{\"o}gelein, Frank Duzaar, Leah Sch{\"a}tzler, and Christoph Scheven.
\newblock Existence for evolutionary problems with linear growth by stability
  methods.
\newblock \emph{Journal of Differential Equations}, 266\penalty0 (11):\penalty0
  7709--7748, May 2019.
\newblock ISSN 0022-0396.
\newblock \doi{10.1016/j.jde.2018.12.012}.
\newblock URL
  \url{https://www.sciencedirect.com/science/article/pii/S0022039618307009}.

\bibitem[Brasco and Lindgren(2017)]{brascoHigherSobolevRegularity2017}
Lorenzo Brasco and Erik Lindgren.
\newblock Higher sobolev regularity for the fractional $p$-laplace equation in
  the superquadratic case.
\newblock \emph{Advances in Mathematics}, 304:\penalty0 300–354, 2017.
\newblock ISSN 0001-8708.
\newblock \doi{10.1016/j.aim.2016.03.039}.
\newblock URL \url{https://mathscinet.ams.org/mathscinet-getitem?mr=3558212}.

\bibitem[Brasco et~al.(2014)Brasco, Lindgren, and
  Parini]{brascoFractionalCheegerProblem2014}
Lorenzo Brasco, Erik Lindgren, and Enea Parini.
\newblock The fractional {{Cheeger}} problem.
\newblock \emph{Interfaces and Free Boundaries}, 16\penalty0 (3):\penalty0
  419--458, September 2014.
\newblock ISSN 1463-9963.
\newblock \doi{10.4171/IFB/325}.
\newblock URL
  \url{https://www.ems-ph.org/journals/show_abstract.php?issn=1463-9963&vol=16&iss=3&rank=5}.

\bibitem[Brasco et~al.(2018)Brasco, Lindgren, and
  Schikorra]{brascoHigherHolderRegularity2018}
Lorenzo Brasco, Erik Lindgren, and Armin Schikorra.
\newblock Higher {{H\"older}} regularity for the fractional p-{{Laplacian}} in
  the superquadratic case.
\newblock \emph{Advances in Mathematics}, 338:\penalty0 782--846, November
  2018.
\newblock ISSN 0001-8708.
\newblock \doi{10.1016/j.aim.2018.09.009}.
\newblock URL
  \url{https://www.sciencedirect.com/science/article/pii/S0001870818303402}.

\bibitem[Brasco et~al.(2021)Brasco, Lindgren, and
  Str{\"o}mqvist]{brascoContinuitySolutionsNonlinear2021}
Lorenzo Brasco, Erik Lindgren, and Martin Str{\"o}mqvist.
\newblock Continuity of solutions to a nonlinear fractional diffusion equation.
\newblock \emph{Journal of Evolution Equations}, June 2021.
\newblock ISSN 1424-3202.
\newblock \doi{10.1007/s00028-021-00721-2}.
\newblock URL \url{https://doi.org/10.1007/s00028-021-00721-2}.

\bibitem[Brezis(2011)]{brezisFunctionalAnalysisSobolev2011}
Haim Brezis.
\newblock \emph{Functional Analysis, {{Sobolev}} Spaces and Partial
  Differential Equations}.
\newblock Universitext. {Springer, New York}, 2011.
\newblock ISBN 978-0-387-70913-0.
\newblock URL \url{https://mathscinet.ams.org/mathscinet-getitem?mr=2759829}.

\bibitem[Byun et~al.(2021)Byun, Ok, and Song]{byunOlderRegularityWeak2021}
Sun-Sig Byun, Jihoon Ok, and Kyeong Song.
\newblock H\"older regularity for weak solutions to nonlocal double phase
  problems.
\newblock \emph{arXiv:2108.09623 [math]}, August 2021.
\newblock URL \url{http://arxiv.org/abs/2108.09623}.

\bibitem[Caffarelli et~al.(2011)Caffarelli, Chan, and
  Vasseur]{caffarelliRegularityTheoryParabolic2011}
Luis Caffarelli, Chi~Hin Chan, and Alexis Vasseur.
\newblock Regularity theory for parabolic nonlinear integral operators.
\newblock \emph{Journal of the American Mathematical Society}, 24\penalty0
  (3):\penalty0 849--869, July 2011.
\newblock ISSN 0894-0347, 1088-6834.
\newblock \doi{10.1090/S0894-0347-2011-00698-X}.
\newblock URL
  \url{https://www.ams.org/jams/2011-24-03/S0894-0347-2011-00698-X/}.

\bibitem[Caffarelli and Vasseur(2010)]{caffarelliDriftDiffusionEquations2010}
Luis~A. Caffarelli and Alexis Vasseur.
\newblock Drift diffusion equations with fractional diffusion and the
  quasi-geostrophic equation.
\newblock \emph{Annals of Mathematics. Second Series}, 171\penalty0
  (3):\penalty0 1903--1930, 2010.
\newblock ISSN 0003-486X.
\newblock \doi{10.4007/annals.2010.171.1903}.
\newblock URL \url{https://mathscinet.ams.org/mathscinet-getitem?mr=2680400}.

\bibitem[Chaker(2020)]{chakerRegularitySolutionsAnisotropic2020}
Jamil Chaker.
\newblock Regularity of solutions to anisotropic nonlocal equations.
\newblock \emph{Mathematische Zeitschrift}, 296\penalty0 (3):\penalty0
  1135--1155, December 2020.
\newblock ISSN 1432-1823.
\newblock \doi{10.1007/s00209-020-02459-y}.
\newblock URL \url{https://doi.org/10.1007/s00209-020-02459-y}.

\bibitem[Chaker and Kassmann(2020)]{chakerNonlocalOperatorsSingular2020}
Jamil Chaker and Moritz Kassmann.
\newblock Nonlocal operators with singular anisotropic kernels.
\newblock \emph{Communications in Partial Differential Equations}, 45\penalty0
  (1):\penalty0 1--31, January 2020.
\newblock ISSN 0360-5302.
\newblock \doi{10.1080/03605302.2019.1651335}.
\newblock URL \url{https://doi.org/10.1080/03605302.2019.1651335}.

\bibitem[Chaker and Kim(2021{\natexlab{a}})]{chakerLocalRegularityNonlocal2021}
Jamil Chaker and Minhyun Kim.
\newblock Local regularity for nonlocal equations with variable exponents.
\newblock \emph{arXiv:2107.06043 [math]}, July 2021{\natexlab{a}}.
\newblock URL \url{http://arxiv.org/abs/2107.06043}.

\bibitem[Chaker and
  Kim(2021{\natexlab{b}})]{chakerRegularityEstimatesFractional2021}
Jamil Chaker and Minhyun Kim.
\newblock Regularity estimates for fractional orthotropic $p$-laplacians of
  mixed order.
\newblock \emph{arXiv:2104.07507 [math]}, April 2021{\natexlab{b}}.
\newblock URL \url{http://arxiv.org/abs/2104.07507}.

\bibitem[Chaker et~al.(2021)Chaker, Kim, and
  Weidner]{chakerRegularityNonlocalProblems2021}
Jamil Chaker, Minhyun Kim, and Marvin Weidner.
\newblock Regularity for nonlocal problems with non-standard growth.
\newblock \emph{arXiv:2111.09182 [math]}, November 2021.
\newblock URL \url{http://arxiv.org/abs/2111.09182}.

\bibitem[{Chang-Lara} and
  D{\'a}vila(2014)]{chang-laraRegularitySolutionsNonlocal2014}
H{\'e}ctor {Chang-Lara} and Gonzalo D{\'a}vila.
\newblock Regularity for solutions of nonlocal parabolic equations {{II}}.
\newblock \emph{Journal of Differential Equations}, 256\penalty0 (1):\penalty0
  130--156, January 2014.
\newblock ISSN 0022-0396.
\newblock \doi{10.1016/j.jde.2013.08.016}.
\newblock URL
  \url{https://www.sciencedirect.com/science/article/pii/S0022039613003859}.

\bibitem[Cozzi(2017)]{cozziRegularityResultsHarnack2017}
Matteo Cozzi.
\newblock Regularity results and {{Harnack}} inequalities for minimizers and
  solutions of nonlocal problems: {{A}} unified approach via fractional {{De
  Giorgi}} classes.
\newblock \emph{Journal of Functional Analysis}, 272\penalty0 (11):\penalty0
  4762--4837, June 2017.
\newblock ISSN 0022-1236.
\newblock \doi{10.1016/j.jfa.2017.02.016}.
\newblock URL
  \url{https://www.sciencedirect.com/science/article/pii/S0022123617300770}.

\bibitem[De~Giorgi(1996)]{degiorgiConjecturesConcerningEvolution1996}
Ennio De~Giorgi.
\newblock Conjectures concerning some evolution problems.
\newblock In \emph{Duke {{Mathematical Journal}}}, volume~81, pages 255--268.
  1996.
\newblock \doi{10.1215/S0012-7094-96-08114-4}.
\newblock URL \url{https://mathscinet.ams.org/mathscinet-getitem?mr=1395405}.

\bibitem[Di~Castro et~al.(2016)Di~Castro, Kuusi, and
  Palatucci]{dicastroLocalBehaviorFractional2016}
Agnese Di~Castro, Tuomo Kuusi, and Giampiero Palatucci.
\newblock Local behavior of fractional p-minimizers.
\newblock \emph{Annales de l'Institut Henri Poincar\'e C, Analyse non
  lin\'eaire}, 33\penalty0 (5):\penalty0 1279--1299, September 2016.
\newblock ISSN 0294-1449.
\newblock \doi{10.1016/j.anihpc.2015.04.003}.
\newblock URL
  \url{https://www.sciencedirect.com/science/article/pii/S0294144915000451}.

\bibitem[Di~Nezza et~al.(2012)Di~Nezza, Palatucci, and
  Valdinoci]{dinezzaHitchhikersGuideFractional2012}
Eleonora Di~Nezza, Giampiero Palatucci, and Enrico Valdinoci.
\newblock Hitchhiker's guide to the fractional {{Sobolev}} spaces.
\newblock \emph{Bulletin des Sciences Math\'ematiques}, 136\penalty0
  (5):\penalty0 521--573, July 2012.
\newblock ISSN 0007-4497.
\newblock \doi{10.1016/j.bulsci.2011.12.004}.
\newblock URL
  \url{https://www.sciencedirect.com/science/article/pii/S0007449711001254}.

\bibitem[Ding et~al.(2021)Ding, Zhang, and
  Zhou]{dingLocalBoundednessHolder2021}
Mengyao Ding, Chao Zhang, and Shulin Zhou.
\newblock Local boundedness and {{H\"older}} continuity for the parabolic
  fractional p-{{Laplace}} equations.
\newblock \emph{Calculus of Variations and Partial Differential Equations},
  60\penalty0 (1):\penalty0 38, January 2021.
\newblock ISSN 1432-0835.
\newblock \doi{10.1007/s00526-020-01870-x}.
\newblock URL \url{https://doi.org/10.1007/s00526-020-01870-x}.

\bibitem[Ghosh et~al.(2022)Ghosh, Kumar, Prasad, and
  Tewary]{ghoshExistenceVariationalSolutions2022}
Suchandan Ghosh, Dharmendra Kumar, Harsh Prasad, and Vivek Tewary.
\newblock Existence of variational solutions to doubly nonlinear nonlocal
  evolution equations via minimizing movements.
\newblock \emph{arXiv:2201.00634 [math]}, January 2022.
\newblock URL \url{http://arxiv.org/abs/2201.00634}.

\bibitem[Han(2022)]{hanCompactSobolevSlobodeckijEmbeddings2022}
Qi~Han.
\newblock Compact {{Sobolev-Slobodeckij}} embeddings and positive solutions to
  fractional {{Laplacian}} equations.
\newblock \emph{Advances in Nonlinear Analysis}, 11\penalty0 (1):\penalty0
  432--453, January 2022.
\newblock ISSN 2191-950X.
\newblock \doi{10.1515/anona-2020-0133}.
\newblock URL
  \url{https://www.degruyter.com/document/doi/10.1515/anona-2020-0133/html}.

\bibitem[Kinnunen and
  Lindqvist(2006)]{kinnunenPointwiseBehaviourSemicontinuous2006}
Juha Kinnunen and Peter Lindqvist.
\newblock Pointwise behaviour of semicontinuous supersolutions to a quasilinear
  parabolic equation.
\newblock \emph{Annali di Matematica Pura ed Applicata}, 185\penalty0
  (3):\penalty0 411--435, August 2006.
\newblock ISSN 1618-1891.
\newblock \doi{10.1007/s10231-005-0160-x}.
\newblock URL \url{https://doi.org/10.1007/s10231-005-0160-x}.

\bibitem[Kinnunen and Masson(2015)]{kinnunenParabolicComparisonPrinciple2015}
Juha Kinnunen and Mathias Masson.
\newblock Parabolic comparison principle and quasiminimizers in metric measure
  spaces.
\newblock \emph{Proceedings of the American Mathematical Society}, 143\penalty0
  (2):\penalty0 621--632, February 2015.
\newblock ISSN 0002-9939, 1088-6826.
\newblock \doi{10.1090/S0002-9939-2014-12236-2}.
\newblock URL
  \url{https://www.ams.org/proc/2015-143-02/S0002-9939-2014-12236-2/}.

\bibitem[Kuusi et~al.(2014)Kuusi, Mingione, and
  Sire]{kuusiFractionalGehringLemma2014}
Tuomo Kuusi, Giuseppe Mingione, and Yannick Sire.
\newblock A fractional {{Gehring}} lemma, with applications to nonlocal
  equations.
\newblock \emph{Rendiconti Lincei - Matematica e Applicazioni}, 25\penalty0
  (4):\penalty0 345--358, November 2014.
\newblock ISSN 1120-6330.
\newblock \doi{10.4171/RLM/683}.
\newblock URL
  \url{https://www.ems-ph.org/journals/show_abstract.php?issn=1120-6330&vol=25&iss=4&rank=1}.

\bibitem[Lara and D{\'a}vila(2014)]{laraRegularitySolutionsNon2014}
H{\'e}ctor~Chang Lara and Gonzalo D{\'a}vila.
\newblock Regularity for solutions of non local parabolic equations.
\newblock \emph{Calculus of Variations and Partial Differential Equations},
  49\penalty0 (1):\penalty0 139--172, January 2014.
\newblock ISSN 1432-0835.
\newblock \doi{10.1007/s00526-012-0576-2}.
\newblock URL \url{https://doi.org/10.1007/s00526-012-0576-2}.

\bibitem[Lichnewsky and
  Temam(1978)]{lichnewskyPseudosolutionsTimedependentMinimal1978}
A~Lichnewsky and R~Temam.
\newblock Pseudosolutions of the time-dependent minimal surface problem.
\newblock \emph{Journal of Differential Equations}, 30\penalty0 (3):\penalty0
  340--364, December 1978.
\newblock ISSN 0022-0396.
\newblock \doi{10.1016/0022-0396(78)90005-0}.
\newblock URL
  \url{https://www.sciencedirect.com/science/article/pii/0022039678900050}.

\bibitem[Marcellini(1989)]{marcelliniRegularityMinimizersIntegrals1989}
Paolo Marcellini.
\newblock Regularity of minimizers of integrals of the calculus of variations
  with nonstandard growth conditions.
\newblock \emph{Archive for Rational Mechanics and Analysis}, 105\penalty0
  (3):\penalty0 267--284, 1989.
\newblock ISSN 0003-9527.
\newblock \doi{10.1007/BF00251503}.
\newblock URL \url{10.1007/BF00251503}.

\bibitem[Marcellini(1991)]{marcelliniRegularityExistenceSolutions1991}
Paolo Marcellini.
\newblock Regularity and existence of solutions of elliptic equations with
  $p,q$-growth conditions.
\newblock \emph{Journal of Differential Equations}, 90\penalty0 (1):\penalty0
  1–30, 1991.
\newblock ISSN 0022-0396.
\newblock \doi{10.1016/0022-0396(91)90158-6}.
\newblock URL \url{10.1016/0022-0396(91)90158-6}.

\bibitem[Marcellini(1993)]{marcelliniRegularityEllipticEquations1993}
Paolo Marcellini.
\newblock Regularity for {{Elliptic Equations}} with {{General Growth
  Conditions}}.
\newblock \emph{Journal of Differential Equations}, 105\penalty0 (2):\penalty0
  296--333, October 1993.
\newblock ISSN 0022-0396.
\newblock \doi{10.1006/jdeq.1993.1091}.
\newblock URL \url{10.1006/jdeq.1993.1091}.

\bibitem[Marcellini(1996)]{marcelliniRegularityScalarVariational1996}
Paolo Marcellini.
\newblock Regularity for some scalar variational problems under general growth
  conditions.
\newblock \emph{Journal of Optimization Theory and Applications}, 90\penalty0
  (1):\penalty0 161--181, July 1996.
\newblock ISSN 1573-2878.
\newblock \doi{10.1007/BF02192251}.
\newblock URL \url{10.1007/BF02192251}.

\bibitem[Marcellini(2020)]{marcelliniRegularityGeneralGrowth2020}
Paolo Marcellini.
\newblock Regularity under general and $ p,q-$ growth conditions.
\newblock \emph{Discrete \& Continuous Dynamical Systems - S}, 13\penalty0
  (7):\penalty0 2009, 2020.
\newblock \doi{10.3934/dcdss.2020155}.
\newblock URL \url{10.3934/dcdss.2020155}.

\bibitem[Menovschikov et~al.(2021)Menovschikov, Molchanova, and
  Scarpa]{menovschikovExtendedVariationalTheory2021}
Alexander Menovschikov, Anastasia Molchanova, and Luca Scarpa.
\newblock An {{Extended Variational Theory}} for {{Nonlinear Evolution
  Equations}} via {{Modular Spaces}}.
\newblock \emph{SIAM Journal on Mathematical Analysis}, 53\penalty0
  (4):\penalty0 4865--4907, January 2021.
\newblock ISSN 0036-1410.
\newblock \doi{10.1137/20M1385251}.
\newblock URL \url{https://epubs.siam.org/doi/10.1137/20M1385251}.

\bibitem[Mingione and R{\u
  a}dulescu(2021)]{mingioneRecentDevelopmentsProblems2021}
Giuseppe Mingione and Vicen{\c t}iu R{\u a}dulescu.
\newblock Recent developments in problems with nonstandard growth and
  nonuniform ellipticity.
\newblock \emph{Journal of Mathematical Analysis and Applications}, page
  125197, March 2021.
\newblock ISSN 0022-247X.
\newblock \doi{10.1016/j.jmaa.2021.125197}.
\newblock URL
  \url{https://www.sciencedirect.com/science/article/pii/S0022247X21002766}.

\bibitem[Parviainen(2008)]{parviainenGlobalHigherIntegrability2008}
Mikko Parviainen.
\newblock Global higher integrability for parabolic quasiminimizers in
  nonsmooth domains.
\newblock \emph{Calculus of Variations and Partial Differential Equations},
  31\penalty0 (1):\penalty0 75--98, January 2008.
\newblock ISSN 1432-0835.
\newblock \doi{10.1007/s00526-007-0106-9}.
\newblock URL \url{https://doi.org/10.1007/s00526-007-0106-9}.

\bibitem[Prasad and Tewary(2021)]{prasadLocalBoundednessVariational2021}
Harsh Prasad and Vivek Tewary.
\newblock Local boundedness of variational solutions to nonlocal double phase
  parabolic equations.
\newblock \emph{arXiv:2112.02345 [math]}, December 2021.
\newblock URL \url{http://arxiv.org/abs/2112.02345}.

\bibitem[Rindler(2018)]{rindlerCalculusVariations2018}
Filip Rindler.
\newblock \emph{Calculus of Variations}.
\newblock Universitext. {Springer, Cham}, 2018.
\newblock ISBN 978-3-319-77636-1 978-3-319-77637-8.
\newblock \doi{10.1007/978-3-319-77637-8}.
\newblock URL \url{https://mathscinet.ams.org/mathscinet-getitem?mr=3821514}.

\bibitem[Scarpa and Stefanelli(2021)]{scarpaStochasticPDEsConvex2021}
Luca Scarpa and Ulisse Stefanelli.
\newblock Stochastic {{PDEs}} via convex minimization.
\newblock \emph{Communications in Partial Differential Equations}, 46\penalty0
  (1):\penalty0 66--97, January 2021.
\newblock ISSN 0360-5302.
\newblock \doi{10.1080/03605302.2020.1831017}.
\newblock URL \url{https://doi.org/10.1080/03605302.2020.1831017}.

\bibitem[Scott and Mengesha(2020)]{scottSelfimprovingInequalitiesBounded2020}
James~M. Scott and Tadele Mengesha.
\newblock Self-improving {{Inequalities}} for bounded weak solutions to
  nonlocal double phase equations.
\newblock \emph{arXiv:2011.11466 [math]}, November 2020.
\newblock URL \url{http://arxiv.org/abs/2011.11466}.

\bibitem[Serra and Tilli(2012)]{serraNonlinearWaveEquations2012}
Enrico Serra and Paolo Tilli.
\newblock Nonlinear wave equations as limits of convex minimization problems:
  Proof of a conjecture by {{De Giorgi}}.
\newblock \emph{Annals of Mathematics. Second Series}, 175\penalty0
  (3):\penalty0 1551--1574, 2012.
\newblock ISSN 0003-486X.
\newblock \doi{10.4007/annals.2012.175.3.11}.
\newblock URL \url{https://mathscinet.ams.org/mathscinet-getitem?mr=2912711}.

\bibitem[Showalter(1997)]{showalterMonotoneOperatorsBanach1997}
R.~E. Showalter.
\newblock \emph{Monotone Operators in {{Banach}} Space and Nonlinear Partial
  Differential Equations}, volume~49 of \emph{Mathematical {{Surveys}} and
  {{Monographs}}}.
\newblock {American Mathematical Society, Providence, RI}, 1997.
\newblock ISBN 978-0-8218-0500-8.
\newblock \doi{10.1090/surv/049}.
\newblock URL \url{https://mathscinet.ams.org/mathscinet-getitem?mr=1422252}.

\bibitem[Stefanelli(2011)]{stefanelliGiorgiConjectureElliptic2011}
Ulisse Stefanelli.
\newblock The {{De Giorgi}} conjecture on elliptic regularization.
\newblock \emph{Mathematical Models and Methods in Applied Sciences},
  21\penalty0 (06):\penalty0 1377--1394, June 2011.
\newblock ISSN 0218-2025.
\newblock \doi{10.1142/S0218202511005350}.
\newblock URL
  \url{https://www.worldscientific.com/doi/abs/10.1142/S0218202511005350}.

\bibitem[Str{\"o}mqvist(2019)]{stromqvistLocalBoundednessSolutions2019}
Martin Str{\"o}mqvist.
\newblock Local boundedness of solutions to non-local parabolic equations
  modeled on the fractional p-{{Laplacian}}.
\newblock \emph{Journal of Differential Equations}, 266\penalty0 (12):\penalty0
  7948--7979, June 2019.
\newblock ISSN 0022-0396.
\newblock \doi{10.1016/j.jde.2018.12.021}.
\newblock URL
  \url{https://www.sciencedirect.com/science/article/pii/S0022039618307095}.

\bibitem[Wieser(1987)]{wieserParabolicQminimaMinimal1987}
Wilfried Wieser.
\newblock Parabolic {{Q-minima}} and minimal solutions to variational flow.
\newblock \emph{Manuscripta Mathematica}, 59\penalty0 (1):\penalty0 63--107,
  March 1987.
\newblock ISSN 1432-1785.
\newblock \doi{10.1007/BF01171265}.
\newblock URL \url{https://doi.org/10.1007/BF01171265}.

\end{thebibliography}

\end{document}